\documentclass[a4paper,11pt]{article}

\usepackage{amsfonts,amssymb,amsmath,amsthm,latexsym}
\usepackage{mathrsfs}
\usepackage{graphicx,graphics}
\usepackage{hyperref}
\usepackage{color}
\usepackage{caption}

\usepackage{subcaption}
\usepackage[english]{babel}
\makeatletter
\def\@captype{figure}
\makeatother

 \makeatletter
 \@addtoreset{equation}{section}
 \makeatother

 \makeatletter
 \@addtoreset{enunciato}{section}
 \makeatother

 \newcounter{enunciato}[section]

 \newtheorem{itassumption}{Assumption}
 \newtheorem{ittheorem}{Theorem}
 \newtheorem{itcorollary}{Corollary}
 \newtheorem{itlemma}{Lemma}
 \newtheorem{itproposition}{Proposition}
 \newtheorem{itdefinition}{Definition}
 \newtheorem{itremark}{Remark}
 \newtheorem{itclaim}{Claim}
 \newtheorem{itfact}{Fact}
 \newtheorem{itconjecture}{Conjecture}
 \newtheorem{itexample}{Example}

 \newenvironment{theorem}{\addtocounter{enunciato}{1}
 \begin{ittheorem}}{\end{ittheorem}}
 
 \newenvironment{corollary}{\addtocounter{enunciato}{1}
 \begin{itcorollary}}{\end{itcorollary}}

 \newenvironment{lemma}{\addtocounter{enunciato}{1}
 \begin{itlemma}}{\end{itlemma}}

 \newenvironment{proposition}{\addtocounter{enunciato}{1}
 \begin{itproposition}}{\end{itproposition}}

 \newenvironment{definition}{\addtocounter{enunciato}{1}
 \begin{itdefinition}}{\end{itdefinition}}

 \newenvironment{remark}{\addtocounter{enunciato}{1}
 \begin{itremark}}{\end{itremark}}

 \newenvironment{claim}{\addtocounter{enunciato}{1}
 \begin{itclaim}}{\end{itclaim}}

 \newenvironment{fact}{\addtocounter{enunciato}{1}
 \begin{itfact}}{\end{itfact}}

 \newenvironment{conjecture}{\addtocounter{enunciato}{1}
 \begin{itconjecture}}{\end{itconjecture}}
 
 \newenvironment{example}{\addtocounter{enunciato}{1}
 \begin{itexample}}{\end{itexample}}
 
 \newcommand{\be}[1]{\begin{equation}\label{#1}}
 \newcommand{\ee}{\end{equation}}

 \newcommand{\bl}[1]{\begin{lemma}\label{#1}}
 \newcommand{\el}{\end{lemma}}

 \newcommand{\br}[1]{\begin{remark}\label{#1}}
 \newcommand{\er}{\end{remark}}

 \newcommand{\bt}[1]{\begin{theorem}\label{#1}}
 \newcommand{\et}{\end{theorem}}

 \newcommand{\bd}[1]{\begin{definition}\label{#1}}
 \newcommand{\ed}{\end{definition}}

 \newcommand{\bcl}[1]{\begin{claim}\label{#1}}
 \newcommand{\ecl}{\end{claim}}

 \newcommand{\bfact}[1]{\begin{fact}\label{#1}}
 \newcommand{\efact}{\end{fact}}

 \newcommand{\bp}[1]{\begin{proposition}\label{#1}}
 \newcommand{\ep}{\end{proposition}}

 \newcommand{\bc}[1]{\begin{corollary}\label{#1}}
 \newcommand{\ec}{\end{corollary}}

 \newcommand{\bcj}[1]{\begin{conjecture}\label{#1}}
 \newcommand{\ecj}{\end{conjecture}}

 \newcommand{\bex}[1]{\begin{example}\label{#1}}
 \newcommand{\eex}{\end{example}}

 \newcommand{\bpr}{\begin{proof}}
 \newcommand{\epr}{\end{proof}}

 \newcommand{\bprt}[1]{\begin{proofoft}{\it\ref{#1}}.\,\,}
 \newcommand{\eprt}{\end{proofoft}}
 
 \newcommand{\bprc}[1]{\begin{proofofc}{\it\ref{#1}}.\,\,}
 \newcommand{\eprc}{\end{proofofc}}
 
 \newcommand{\bprp}[1]{\begin{proofofp}{\it\ref{#1}}.\,\,}
 \newcommand{\eprp}{\end{proofofp}}
 
 \newcommand{\bprl}[1]{\begin{proofofl}{\it\ref{#1}}.\,\,}
 \newcommand{\eprl}{\end{proofofl}}

 \newcommand{\bi}{\begin{itemize}}
 \newcommand{\ei}{\end{itemize}}

 \newcommand{\ben}{\begin{enumerate}}
 \newcommand{\een}{\end{enumerate}}
 
\newenvironment{proofoft}{\noindent 
{\it Proof of Theorem\,}}{\hspace*{\fill}$\halmos$\medskip}
\newenvironment{proofofc}{\noindent 
{\it Proof of Corollary\,}}{\hspace*{\fill}$\halmos$\medskip}
\newenvironment{proofofp}{\noindent 
{\it Proof of Proposition\,}}{\hspace*{\fill}$\halmos$\medskip}
\newenvironment{proofofl}{\noindent 
{\it Proof of Lemma\,}}{\hspace*{\fill}$\halmos$\medskip}
\newcommand{\halmos}{\rule{1ex}{1.4ex}}

\def \PP {{\mathbb P}}
\def \EE {{\mathbb E}}

\def \Sp {{\rm Sp}}
\def \Supp {{\rm Supp}}

\def \E {{\mathbb E}}
\def \N {{\mathbb N}}
\def \P {{\mathbb P}}
\def \Q {{\mathbb Q}}
\def \R {{\mathbb R}}
\def \Z {{\mathbb Z}}

\def \cE {{\mathcal E}}

\def \cK {{\mathcal K}}

\def \cH {{\mathcal H}}
\def \cL {{\mathcal L}}
\def \cM {{\mathcal M}}

\def \cF {{\mathcal F}}
\def \WN {{\hbox{\tiny\rm STWN}}}
\def \FSRW {{\hbox{\tiny\rm FIRW}}}
\def \ISRW {{\hbox{\tiny\rm IIRW}}}
\def \Sp {{\hbox{\tiny\rm SFS}}}

\newcommand{\one}{{\mathchoice {1\mskip-4mu\mathrm l}
{1\mskip-4mu\mathrm l}
{1\mskip-4.5mu\mathrm l}
{1\mskip-5mu\mathrm l}}}

\newcommand{\esssup}{\mathrm{ess\,sup}}
\newcommand{\essinf}{\mathrm{ess\,inf}}

\begin{document}

\title{Parabolic Anderson model in a dynamic random environment:\\
random conductances}

\author{\renewcommand{\thefootnote}{\arabic{footnote}}
D.\ Erhard
\footnotemark[1]
\\
\renewcommand{\thefootnote}{\arabic{footnote}}
F.\ den Hollander
\footnotemark[2]
\\
\renewcommand{\thefootnote}{\arabic{footnote}}
G.\ Maillard
\footnotemark[3]
}

\footnotetext[1]{
Mathematics Department, University of Warwick,
Coventry, CV4 7AL, United Kingdom,\\
{\sl D.Erhard@warwick.ac.uk}
}
\footnotetext[2]{
Mathematical Institute, Leiden University, P.O.\ Box 9512,
2300 RA Leiden, The Netherlands,\\
{\sl denholla@math.leidenuniv.nl}
}
\footnotetext[3]{
Aix-Marseille Universit\'e, CNRS, Centrale Marseille, I2M, 
UMR 7373, 13453 Marseille, France,\\
{\sl gregory.maillard@univ-amu.fr}
}

\date{\today}
\maketitle


\begin{abstract}
The parabolic Anderson model is defined as the partial differential equation $\partial 
u(x,t)/\partial t$ $= \kappa\Delta u(x,t) + \xi(x,t)u(x,t)$, $x\in\Z^d$, $t\geq 0$, where $\kappa 
\in [0,\infty)$ is the diffusion constant, $\Delta$ is the discrete Laplacian, and $\xi$ is 
a \emph{dynamic random environment} that drives the equation. The initial condition 
$u(x,0)=u_0(x)$, $x\in\Z^d$, is typically taken to be non-negative and bounded. The 
solution of the parabolic Anderson equation describes the evolution of a field of 
particles performing independent simple random walks with binary branching: particles 
jump at rate $2d\kappa$, split into two at rate $\xi \vee 0$, and die at rate $(-\xi) \vee 0$. 
In earlier work we looked at the \emph{Lyapunov exponents}
\[
\lambda_p(\kappa) = \lim_{t\to\infty} \frac{1}{t} \log \E([u(0,t)]^p)^{1/p}, \quad p \in \N,  
\qquad \lambda_0(\kappa) = \lim_{t\to\infty} \frac{1}{t}\log u(0,t).
\]
For the former we derived \emph{quantitative} results on the $\kappa$-dependence for four 
choices of $\xi$: space-time white noise, independent simple random walks, the exclusion 
process and the voter model. For the latter we obtained \emph{qualitative} results under 
certain space-time mixing conditions on $\xi$.

In the present paper we investigate what happens when $\kappa\Delta$ is replaced 
by $\Delta^\cK$, where $\cK = \{\mathcal{K}(x,y)\colon\,x,y\in\Z^d,\,x \sim y\}$ is a collection of random
conductances between neighbouring sites replacing the constant conductances $\kappa$ 
in the homogeneous model. We show that the associated \emph{annealed} Lyapunov 
exponents $\lambda_p(\cK)$, $p\in\N$, are given by the formula
\[
\lambda_p(\cK) = \sup\{\lambda_p(\kappa) \colon\,\kappa\in\Supp(\cK)\}, 
\]
where $\Supp(\cK)$ is the set of values taken by the $\cK$-field. We also show that for 
the associated \emph{quenched} Lyapunov exponent $\lambda_0(\cK)$ this formula 
only provides a lower bound, and we conjecture that an upper bound holds when $\Supp(\cK)$ 
is replaced by its convex hull. Our proof is valid for three classes of reversible $\xi$, and for 
all $\cK$ satisfying a certain \emph{clustering property}, namely, there are arbitrarily large 
balls where $\cK$ is almost constant and close to any value in $\Supp(\cK)$. What our result 
says is that the Lyapunov exponents are controlled by those pockets of $\cK$ where the 
conductances are close to the value that maximises the growth in the homogeneous setting. 
Our proof is based on variational representations and confinement arguments showing that mixed 
pockets are subdominant. 

\medskip\noindent
{\it MSC 2010.} Primary 60K35, 60H25, 82C44; Secondary 35B40, 60F10.\\
{\it Key words and phrases.} Parabolic Anderson equation, random conductances, Lyapunov 
exponents, large deviations, variational representations, confinement.\\
{\it Acknowledgment.} The authors were supported by ERC Advanced Grant 267356 
VARIS of FdH. DE was also supported by ERC Consolidator Grant 615897 CRITICAL of 
Martin Hairer. GM is grateful to the Mathematical Institute of Leiden University for
hospitality during sabbatical leaves in May-July 2014 and in June-July 2015.
\end{abstract}


\section{Introduction and main results}
\label{S1}

Random walks with random conductances have been studied intensively in the literature.
For a recent overview, we refer the reader to Biskup~\cite{B11}. The goal of the present paper 
is to study the version of the Parabolic Anderson model where the underlying random walk 
is driven by random conductances, and to investigate the effect on the Lyapunov exponents.  


\subsection{Parabolic Anderson model with random conductances}
\label{S1.1}

The parabolic Anderson model with random conductances is the partial differential equation
\begin{equation}
\label{pA}
\left\{
\begin{array}{l}
\displaystyle\frac{\partial}{\partial t}u(x,t) = (\Delta^\cK u)(x,t) + \xi(x,t)u(x,t),\bigg.\\
u(x,0) = u_0(x),
\end{array}
\qquad x\in\Z^d,\,t\geq 0\, .
\right.
\end{equation}
Here, $u$ is an $\R$-valued random field, $\Delta^\cK$ is the discrete Laplacian with 
random conductances $\cK$ acting on $u$ as
\begin{equation}
\label{dL}
\Delta^\cK u(x,t) = \sum_{{y\in\Z^d} \atop {y \sim x}} \cK(x,y) [u(y,t)-u(x,t)],
\end{equation}
where $\{\cK(x,y)\colon x,y\in\Z^d, x \sim y\}$ is a $(0,\infty)$-valued field of \emph{random
conductances}, $x \sim y$ means that $x$ and $y$ are neighbours, while
\begin{equation}
\label{rf}
\xi = (\xi_t)_{t \geq 0} \mbox{ with } \xi_t = \{\xi(x,t) \colon\,x\in\Z^d\}
\end{equation}
is an $\R$-valued random field playing the role of a \emph{dynamic random environment}
that drives the equation. Throughout the paper we assume that 
\begin{equation}
0 \leq u_0(x) \leq 1 \qquad \forall\,x \in \Z^d.
\end{equation}

The $\xi$-field and the $\cK$-field are defined on probability spaces $(\Omega,\cF,\PP)$ 
and $(\tilde\Omega,\tilde\cF,\tilde\PP)$, respectively. Throughout the paper we assume 
that 
\begin{equation}
\label{staterg}
\begin{array}{ll}
{\rm (1)} &0<c \leq \cK(x,y) \leq C<\infty\,\,\forall\,x,y\in\Z^d,\,x \sim y.\\
{\rm (2)} &\cK(x,y)=\cK(y,x)\,\,\forall\,x,y\in\Z^d,\,x \sim y.
\end{array}
\end{equation}
The formal solution of \eqref{pA} is given by the \emph{Feynman-Kac formula}
\begin{equation}
\label{FK}
u(x,t) = E_x\left(\exp\left\{\int_0^t \xi\big(X^\cK(s),t-s\big)\,ds\right\}\,
u_0(X^\cK(t))\right),
\end{equation}
where $X^\cK=(X^\cK(t))_{t \geq 0}$ is the continuous-time Markov process with generator 
$\Delta^\cK$, and $P_x$ is the law of $X^\cK$ given $X^\cK(0)=x$. When $\cK \equiv \kappa
\in (0,\infty)$, we write $X^\cK = X^\kappa$. In Section~\ref{S1.3} we will show that under 
mild assumptions on $\xi$ the formula in \eqref{FK} is the unique non-negative solution of 
\eqref{pA}. These assumptions are fulfilled for the three classes of $\xi$ that will receive special 
attention in our paper, which we list next.


\subsection{Choices of dynamic random environments} 
\label{S1.2}

{\bf (I) Space-time white noise:}
Here $\xi$ is the Markov process on $\Omega = \R^{\Z^d}$  given by
\begin{equation}
\label{eq:wn}
\xi(x,t) = \frac{\partial}{\partial t} W(x,t),
\end{equation}
where $W=(W_t)_{t \geq 0}$ with $W_t = \{W(x,t)\colon\,x\in\Z^d\}$ is a field of independent 
standard Brownian motions, and \eqref{pA} is to be understood as an It\^{o}-equation.

\medskip\noindent
{\bf (II) Independent random walks:}

\medskip\noindent
{\bf (IIa) Finite system:}
Here $\xi$ is the Markov process on $\Omega = \{0,\ldots,n\}^{\Z^d}$ given by
\begin{equation}
\label{eq:fisrw}
\xi(x,t) = \sum_{k=1}^{n}\delta_x(Y_k^{\rho}(t)),
\end{equation}
where $\{Y_k^{\rho}\colon\,1\leq k\leq n\}$ is a collection of $n\in\N$ independent continuous-time
simple random walks jumping at rate $2d\rho$ and starting at the origin.

\medskip\noindent
{\bf (IIb) Infinite system:}
Here $\xi$ is the Markov process on $\Omega = \N_0^{\Z^d}$ given by
\begin{equation}
\label{eq:isrw}
\xi(x,t) = \sum_{y\in\Z^d}\sum_{j=1}^{N_y} \delta_{x}(Y_j^{y}(t)),
\end{equation}
where $\{Y_j^{y}\colon\,y\in\Z^d,1\leq j\leq N_y, Y_j^{y}(0)=y\}$ is an infinite collection of 
independent continuous-time simple random walks jumping at rate $2d$, and $(N_y)_{y\in\Z^d}$ 
is a Poisson random field with intensity $\nu\in(0,\infty)$. The generator $L$ of this process is 
defined as follows (see Andjel~\cite{A82}). Let $l(x)=e^{-\|x\|}$, $x\in\Z^d$, with $\|\cdot\|$
the Euclidean norm. Define the $l$-norm on $\Omega$ as
\begin{equation}
\label{eq:norml}
\|\eta\|_l = \sum_{x\in\Z^d} \eta(x)l(x),
\end{equation}
and define the sets $\mathcal{E}_l= \{\eta\in\Omega\colon\, \|\eta\|_l<\infty\}$ and $\mathcal{L}_l
=\{f\colon\,\cE_l\to\R \mbox{ Lipschitz continuous}\}$. Then $L$ acts on $f\in\cL_l$ as
\begin{equation}
\label{eq:genISRW}
(Lf)(\eta) = \sum_{x\in\Z^d} \sum_{ {y\in\Z^d} \atop {y \sim x} }
\eta(x)[f(\eta^{x,y})-f(\eta)],
\end{equation}
and $\eta^{x,y}$ is defined by
\begin{equation}
\label{eq:etaxy}
\eta^{x,y}(z) = 
\left\{\begin{array}{ll}
\eta(z),\, &z\neq x,y,\\
\eta(x)-1,\, &z=x,\\
\eta(y)+1,\, &z=y.
\end{array}
\right.
\end{equation}
Write $\mu$ for the Poisson random field with intensity $\nu$. This is the invariant
distribution of the dynamics.

\medskip\noindent
{\bf (III) Spin-flip systems:}
Here $\xi$ is the Markov process on $\Omega=\{0,1\}^{\Z^d}$ whose generator $L$ acts 
on cylinder functions $f$ as (see Liggett \cite[Chapter III]{L85})
\be{Lss}
(Lf)(\eta) = \sum_{x\in\Z^d} c(x,\eta)[f(\eta^{x})-f(\eta)],
\ee 
where, for a configuration $\eta$, $c(x,\eta)$ is the rate for the spin at $x$ to flip, and
\begin{equation}
\label{eq:etax}
\eta^x(z) =
\left\{\begin{array}{ll}
\eta(z),\, &z \neq x,\\ 
1-\eta(x),\, &z=x.
\end{array}
\right.
\end{equation}
We assume that the rates $c(x,\eta)$ are such that
\begin{itemize}
\item[(i)] 
$\xi$ is ergodic and \emph{reversible}, i.e., there is a probability distribution $\mu$ on $\Omega$
such that $\xi_t$ converges to $\mu$ in distribution as $t\to\infty$ for any choice of $\xi_0 \in 
\Omega$, and $c(x,\eta)\mu(d\eta) = c(x,\eta^x)\mu(d\eta^x)$ for all $\eta\in\Omega$ and 
$x \in \Z^d$.
\item[(ii)] 
$\xi$ is \emph{attractive}, i.e., $c(x,\eta) \leq c(x,\zeta)$ for all $\eta \leq \zeta$ when 
$\eta(x) = \zeta(x)=0$ and $c(x,\eta) \geq c(x,\zeta)$ for all $\eta \leq \zeta$ when 
$\eta(x) = \zeta(x)=1$ (where we write $\eta \leq \zeta$ when $\eta(x) \leq \zeta(x)$ for 
all $x\in\Z^d$).
\end{itemize}
We further assume that
\begin{itemize}
\item[(iii)] $\xi_0$ has distribution $\mu$.
\end{itemize}
Let  $\cM$ be the class of continuous non-decreasing functions $f$ on $\Omega$, the 
latter meaning that $f(\eta) \leq f(\zeta)$ for all $\eta \leq \zeta$. As shown in 
Liggett~\cite[Theorems II.2.14 and III.2.13]{L85}, attractive spin-flip systems preserve the 
FKG-inequality, i.e., if $\xi_0$ satisfies the FKG-inequality (e.g. if $\xi_0$ is distributed according to $\mu$), then so does $\xi_t$ for all $t\geq 0$, i.e., 
\begin{equation}
\label{eq:poscor}
\EE(f(\xi_t)g(\xi_t)) \geq \EE(f(\xi_t))\,\EE(g(\xi_t)) \quad \forall\,\,f,g\in \cM.
\end{equation}
Examples include the ferromagnetic stochastic Ising model, for which
\begin{equation}
\label{eq:sim}
c(x,\eta)= \exp\left[-\beta \sum_{ {y \in \Z^d} \atop {y \sim x} } \sigma(x)\sigma(y)\right],
\qquad \sigma(x) = 2\eta(x)-1 \in \{-1,+1\},
\end{equation}
with $\beta \in (0,\infty)$ the inverse temperature. This dynamics has at least one invariant 
distribution. It is shown in Liggett~\cite[Theorem IV.2.3 and Proposition IV.2.7]{L85} 
that any reversible spin-flip system is a stochastic Ising model for some interaction 
potential (not necessarily between neighbours).


\subsection{Lyapunov exponents}
\label{S1.3}

Our focus will be on the \emph{annealed Lyapunov exponents}
\be{aLyapdef}
\lambda_p(\cK) = \lim_{t\to\infty} \frac{1}{t} \log\EE\big([u(0,t)]^p\big)^{1/p},
\qquad p \in \N,
\ee
and the \emph{quenched Lyapunov exponent}
\be{qLyapdef}
\lambda_0(\cK) = \lim_{t\to\infty} \frac{1}{t} \log u(0,t),
\ee
provided the limits exist. Note that 
\begin{itemize}
\item[$\blacktriangleright$] $\cK$ is \emph{fixed}, i.e., the annealing and the 
quenching is with respect to $\xi$ only. 
\end{itemize}
We write $\lambda_p(\kappa)$ when $\cK\equiv \kappa$.

Let $\cE^d$ be the edge set of $\Z^d$, and let $\mathrm{Supp}(\cK) = \{\cK(x,y)\colon\,(x,y)
\in\cE^d\}$. For $x\in\Z^d$ and $t>0$, let 
\be{boxdef}
B_{t}(x) = x + ([-t,t]^d \cap \cE^d)
\ee 
be the edges in the box of radius $t$ centered at $x$.

\begin{definition}
\label{def:clustering}
We say that $\cK$ has the clustering property when for all $\kappa \in \Supp(\cK)$, $\delta>0$ 
and $t>0$ there exist radii $L_{\delta,\kappa}(t)$, satisfying $\lim_{t\to\infty} L_{\delta,\kappa}(t) 
= \infty$, and centers $x(\kappa,\delta,t) \in \Z^d$, satisfying $\lim_{t\to\infty} \|x(\kappa,\delta,t)\|/t 
= 0$, such that $\cK(y,z) \in (\kappa-\delta,\kappa+\delta) \cap \Supp(\cK)$ for all $(y,z)
\in B_{L_{\delta,\kappa}(t)}(x(\kappa,\delta,t))$.
\end{definition}

\noindent
For the binary case $\Supp(\cK)=\{\kappa_1,\kappa_2\}$, the clustering property states 
that there are two sequences of boxes $B^1(t)$ and $B^2(t)$, whose sizes tend to infinity 
and whose distances to the origin are $o(t)$, such that $\cK(x,y) = \kappa_1$ for all $(x,y) 
\in B^1(t)$ and $\cK(x,y)=\kappa_2$ for all $(x,y)\in B^2(t)$. Note that if $\cK$ is i.i.d., then 
it has the clustering property with probability $1$.

Our main result for the annealed Lyapunov exponents is the following. 

\bt{th:ann}
Let $\xi$ be as in {\rm {\bf (I)--(III)}}, and let $\cK$ have the clustering property. Then for 
all $p\in\N$ the limit in \eqref{aLyapdef} exists and equals
\be{annequal}
\lambda_p(\cK) = \sup\{\lambda_{p}(\kappa)\colon \kappa\in \Supp(\cK)\},
\qquad p \in \N.
\ee
This equality holds irrespective of whether the right-hand side is finite or infinite. Moreover, 
$\lambda_p(\cK)$ is continuous, non-decreasing and convex in each of the components 
of $\cK$ on any open domain where it is finite.  
\et

To obtain a similar result for the quenched Lyapunov exponent, we need to make a different 
set of assumptions on $\xi$: 
\begin{itemize}
\item[(1)] 
$\xi$ is \emph{stationary} and \emph{ergodic} under translations in space and time.
\item[(2)] 
$\xi$ is \emph{not constant} and  $\E(|\xi(0,0)|)<\infty$.
\item[(3)] 
$s \mapsto \xi(x,s)$ is locally integrable for every $x\in\Z^d$, $\xi$-a.s.
\item[(4)] 
$\E(e^{q\xi(0,0)})<\infty$ for all $q \in \R$.
\end{itemize}
As a consequence of Assumptions (1)--(4), \eqref{pA} has a unique non-negative solution 
given by \eqref{FK} (see Erhard, den Hollander and Maillard~\cite{EdHM12}). The dynamics 
in {\bf (I)--(III)} satisfy (1)--(4). More examples may be found in \cite[Corollary 1.19]{EdHM12}. 

\bt{th:que}
Suppose that $u(x,0) = \delta_0(x)$. Let $\xi$ satisfy {\rm (1)--(4)}, and let $\cK$ have the 
clustering property. Then the limit in \eqref{qLyapdef} exists~$\P$-a.s. and in $\P$-mean 
and satisfies
\be{queequal}
\lambda_0(\cK) \geq \sup\{\lambda_{0}(\kappa)\colon \kappa\in \Supp(\cK)\}. 
\ee
This inequality holds irrespective of whether the right-hand side is finite or infinite. 
\et


\subsection{Discussion and outline}
\label{S1.4}

{\bf 1.}
Theorem \ref{th:ann} shows that, in the annealed setting, the clustering strategy wins over 
the non-clustering strategy, i.e., the annealed Lyapunov exponents are controlled by those 
pockets in $\cK$ where the conductances are close to the value that maximises the growth 
in the homogeneous setting, i.e., mixed pockets in $\cK$ are subdominant. For the quenched
Lyapunov exponent this is not expected to be the case. For the annealed Lyapunov exponents we can use 
variational representations, for the quenched Lyapunov exponent the argument is more delicate.
 
\medskip\noindent  
{\bf 2.}
Examples {\bf (I)} and {\bf (III)} are \emph{non-conservative} dynamics. Examples {\bf (IIa)--(IIb)} 
are \emph{conservative} dynamics.  All are \emph{reversible}.
  
\medskip\noindent  
{\bf 3.} 
For $\cK \equiv \kappa$, the annealed Lyapunov exponents $\lambda_p(\kappa)$, $p\in\N$, 
are known to be continuous, non-increasing and convex in $\kappa$ when finite, for each of 
the choices in {\bf (I)--(III)}. Hence \eqref{annequal} reduces to
\be{annequalalt}
\lambda_p(\cK) = \lambda_{p}(\kappa_*),
\qquad \kappa_* = \essinf [\Supp(\cK)],
\qquad p \in \N,
\ee
i.e., the annealed growth is dominated by the pockets with the \emph{slowest} conductances. 

\medskip\noindent
{\bf 4.} 
The quenched Lyapunov exponent $\lambda_0(\kappa)$ is continuous in $\kappa$ as well, 
but it fails to be non-increasing (it is expected to be unimodal). Hence we do not expect the 
inequality in \eqref{queequal} to be an equality, as in the annealed case. In Section~\ref{S5} 
we provide an illustrative counterexample for a \emph{decorated} version of $\Z^d$, i.e., 
each pair of neighbouring sites of $\Z^d$ is connected by two edges rather than one, for 
which the inequality in \eqref{queequal} is strict. We conjecture that the following upper 
bound holds.

\bcj{conj:queub}
Under the conditions of Theorem~{\rm \ref{th:que}},
\be{}
\lambda_0(\cK) \leq \sup\{\lambda_{0}(\kappa)\colon \kappa\in \mathrm{Conv}(\Supp(\cK))\}, 
\ee
where $\mathrm{Conv}(\Supp(\cK))$ is the convex hull of $\Supp(\cK)$.
\ecj

\medskip\noindent
{\bf 5.}
The Feynman-Kac formula shows that understanding the Lyapunov exponents amounts to
understanding the large deviation behaviour of the integral of the $\xi$-field along the trajectory 
of a random walk in random environment. Drewitz~\cite{D08} studies the case where $\Delta$ 
is replaced by a Laplacian with a deterministic drift and $\xi$ is constant in time. It is proven that 
the Lyapunov exponent is maximal when the drift is zero.

\medskip\noindent
{\bf Outline.}
The outline of the remainder of the paper is as follows. 
In Section~\ref{S2} we derive variational formulas for the annealed Lyapunov exponents 
and use these to derive the rightmost inequality in \eqref{eq:upperboundchain}, i.e., $\leq$ 
in \eqref{annequal}. In Section~\ref{S3} we derive the leftmost inequality in 
\eqref{eq:upperboundchain}, i.e., $\geq$  in \eqref{annequal}. The proof uses a 
\emph{confinement approximation}, showing that the annealed Lyapunov exponent does 
not change when the random walk in the Feynman-Kac formula \eqref{FK} is confined to 
a slowly growing box. In Section~\ref{S4} we turn to the quenched Lyapunov exponent 
and prove the lower bound in Theorem~\ref{th:que} with the help of a confinement 
approximation. In Section~\ref{S5} we discuss the failure of the corresponding upper 
bound by providing a counterexample for a decorated lattice.

In Appendix~\ref{App} we show that the annealed Lyapunov exponents are the same for 
all initial conditions that are bounded. In Appendix~\ref{SB} we prove a technical lemma 
about the generator of dynamics {\bf (IIb)}.


\section{Annealed Lyapunov exponents: preparatory facts, variational representations, 
existence and upper bound}
\label{S2}

Section~\ref{S2.1} contains some preparatory facts.
Section~\ref{S2.2} gives variational representations for $\lambda_p(\cK)$ for each of the four 
dynamics (Propositions~\ref{prop:wn}--\ref{prop:revss} below) and settles the existence.
Section~\ref{S2.3} explains why these variational representations imply the upper bound. 
Section~\ref{S2.4} provides the proof of the variational representations.


\subsection{Preparatory facts}
\label{S2.1}

The following proposition, whose proof is deferred to Appendix \ref{App}, shows that 
the annealed Lyapunov exponents are the same for any bounded initial condition 
$u_0$, i.e., without loss of generality we may take $u_0=\delta_0$ or $u_0\equiv 1$.

\begin{proposition}
\label{prop:locin}
Fix $p\in\N$ and $\kappa>0$. Let $\xi$ be as in {\bf (I)-(III)}, and let 
$\lambda_p^{\delta_0}(\kappa)$ and $\lambda_p^{\one}
(\kappa)$ be the $p$-th annealed Lyapunov exponent for $u_0=\delta_0$ 
and $u_0 \equiv 1$, respectively. 
Then
\begin{equation}
\label{eq:loc}
\lambda_p^{\delta_0}(\kappa)= \lambda_p^{\one}(\kappa).
\end{equation}
\end{proposition}
Consequently, the proof of Theorem~\ref{th:ann} reduces to the 
following two inequalities:
\begin{equation}
\label{eq:upperboundchain}
\begin{aligned}
\sup\{\lambda_p^{\delta_0}(\kappa)\colon\, \kappa\in\mathrm{Supp}(\mathcal{K})\}
\leq \lambda_p^{\delta_0}(\mathcal{K}),
\qquad
\lambda_p^{\one}(\mathcal{K})\leq \sup\{\lambda_p^{\one}(\kappa)\colon\, 
\kappa\in\mathrm{Supp}(\mathcal{K})\}.
\end{aligned}
\end{equation} 
We prove the second inequality (upper bound) in the present section and the 
first inequality (lower bound) in Section~\ref{S3}. For ease of notation we 
suppress the upper index from the respective Lyapunov exponents.

Before we proceed we make three observations:

\medskip\noindent
{\bf(I)} For $\xi$ space-time white noise, it follows from Carmona and 
Molchanov~\cite[Theorem II.3.2]{CM94} that 
\begin{equation}
\label{eq:pmwn}
\EE([u(0,t)]^p) = E_0^{\otimes p}
\Bigg(\exp\Bigg\{\sum_{1\leq i<j\leq p}
\int_0^{t} \one\{X_i^{\cK}(s)=X_j^{\cK}(s)\}\, ds\Bigg\}\prod_{i=1}^{p}
u_0(X_i^{\cK}(t))\Bigg),
\end{equation}
where $E_0^{\otimes p}$ is the expectation with respect to $p$ independent simple 
random walks $X_1^{\cK},\ldots, X_p^{\cK}$, all having generator $\Delta^{\cK}$ and 
all starting at $0$. 

\medskip\noindent
{\bf(IIa)} For $\xi$ finite independent simple random walks we have
\begin{equation}
\label{eq:pmfnirw}
\EE([u(0,t)]^p) = (E_0^{\otimes p}\otimes E_0^{\otimes n})
\Bigg(\exp\Bigg\{\sum_{i=1}^{p}\sum_{j=1}^{n}
\int_0^{t} \one\{X_i^{\cK}(s)=X_j^{\rho}(s)\}\, ds\Bigg\}\prod_{i=1}^{p}
u_0(X_i^{\cK}(t))\Bigg),
\end{equation}
which is similar to \eqref{eq:pmwn}. In particular, the proof of the upper bound in Theorem~\ref{th:ann} 
is similar for {\bf(I)} and {\bf(IIa)}. Therefore we will only give the proof for {\bf(IIa)}.

\medskip\noindent
{\bf(I)--(III)} are reversible, and so we have
\begin{equation}
\label{eq:timerev}
u(0,t) = E_0\bigg(\exp\bigg\{\int_0^t \xi(X^{\cK}(s),s)\, ds\bigg\} u_0(X^{\cK}(t))\bigg)
\end{equation}
in $\P$-distribution. 


\subsection{Variational representations}
\label{S2.2}

We assume that $u_{0}\equiv 1$. For $p\in\N$, $i\in\{1,\dots,p\}$, 
$x=(x_{1},\dots,x_{p}) \in \Z^{dp}$ and $y\in\Z^d$, write $f(x)|_{x_i\to y}$ 
to denote $f(x)$ but with the argument $x_i$ replaced by $y$. 

\bp{prop:wn}
Let $\xi$ be as in {\bf (I)}. Then, for all $p\in\N$,
\be{varwn}
\lambda_{p}(\cK)=\frac1p\sup_{\|f\|_{l^2(\Z^{dp})}=1} \{A_{1}(f)-A_{2}(f)\},
\ee
where
\be{Dirwn}
\begin{aligned}
A_1(f) &= \sum_{x\in\Z^{dp}}
\sum_{1\leq i<j\leq p} \delta_{0}(x_{i},x_{j}) f(x)^2,\\[0.3cm]
A_2(f) &=  \frac12 \sum_{x\in\Z^{dp}}
\sum_{i=1}^p \sum_{{z\in\Z^d} \atop {z \sim x_i}}
\cK(x_{i},z)\big[f(x)|_{x_i\to z}-f(x)\big]^2.
\end{aligned}
\ee
\ep

\bp{prop:fnirw}
Let $\xi$ be as in {\bf (IIa)}. Then, for all $p\in\N$,
\be{varfnirw}
\lambda_{p}(\cK)=\frac1p\sup_{\|f\|_{l^2(\Z^{dp} \times \Z^{dn})}=1} \{A_{1}(f)-A_{2}(f)-A_{3}(f)\},
\ee
where
\be{Dirfnirw}
\begin{aligned}
A_1(f) &= \sum_{x\in\Z^{dp}}\sum_{y\in\Z^{dn}}
\sum_{i=1}^{p}\sum_{j=1}^{n} \delta_{0}(x_{i},y_{j}) f(x,y)^2,\\[0.3cm]
A_2(f) &=  \frac12 \sum_{x\in\Z^{dp}}\sum_{y\in\Z^{dn}}
\sum_{i=1}^p \sum_{{z\in\Z^d} \atop {z \sim x_i}}
\cK(x_{i},z)\big[f(x,y)|_{x_i\to z}-f(x,y)\big]^2,\\[0.3cm]
A_3(f) &= \frac{\rho}{2} \sum_{x\in\Z^{dp}}\sum_{y\in\Z^{dn}}
\sum_{j=1}^n \sum_{{z\in\Z^d} \atop {z \sim y_j}}
\big[f(x,y)|_{y_j\to z}-f(x,y)\big]^2.
\end{aligned}
\ee
\ep

\bp{prop:infirw}
Fix $p\in\N$. Let $\xi$ be as in {\bf (IIb)} and let $G(0)$ be the Green function at the origin of simple random walk jumping at rate $2d$. Then, for all $0<p<1/G(0)$,
\begin{equation}
\label{eq:varISRW}
\lambda_p(\cK) = \frac{1}{p}\sup_{N\in\N}\,\,\sup_{\|f\|_{L^2(\mu\otimes m)=1}}
\left\langle \left(L+\sum_{i=1}^{p} \Delta_i^{\cK}+V_N\right)f,f\right\rangle,
\end{equation}
where
\be{Lpdef}
(\Delta_i^{\cK} f) (\eta,y) = \sum_{ {z\in\Z^d} \atop {z \sim y_i}}
\cK(y_i,z)\big[f(\eta,y)|_{y_i\to z}-f(\eta,y)\big],
\ee
$\mu = \otimes_{i\in\Z^d}\mathrm{POI}(\nu)$ is the Poisson random field with intensity $\nu \in 
(0,\infty)$, $m$ is the counting measure on $\Z^d$, and $V_N\colon\,\N_0^{\Z^d}\times \Z^{pd}\to\R$ 
is the truncated function given by 
\be{VNdef}
V_N(\eta,x) = \sum_{i=1}^{p} [N \wedge\eta(x_i)]
\ee
and $L$ acts on $f$ solely on its first coordinate.
\ep

\bp{prop:revss}
Let $\xi$ be as in {\bf (III)}. Then, for all $p\in\N$,
\be{varrevss}
\lambda_{p}(\cK)=\frac1p\sup_{\|f\|_{L^2(\mu\otimes m^p)}=1} \{A_{1}(f)-A_{2}(f)-A_{3}(f)\},
\ee
where $m^p$ is the counting measure on $\Z^{dp}$, and
\be{Dirrevss}
\begin{aligned}
A_1(f) &= \int_\Omega \mu(d\eta) \sum_{x\in\Z^{dp}}
\sum_{i=1}^{p} \eta(x_i)\,f(\eta,x)^2,\\[0.3cm]
A_2(f) &= \frac12 \int_\Omega \mu(d\eta) \sum_{x\in\Z^{dp}}
\sum_{i=1}^p \sum_{{y\in\Z^d} \atop {y \sim x_i}}
\cK(x_{i},y)\big[f(\eta,x)|_{x_i\to y}-f(\eta,x)\big]^2,\\[0.3cm]
A_3(f) &= \frac12 \int_\Omega \mu(d\eta) \sum_{x\in\Z^{dp}}
\sum_{y\in\Z^d}
c(y,\eta)\,\big[f(\eta^{y},x)-f(\eta,x)\big]^2.
\end{aligned}
\ee
\ep


\subsection{Proof of the upper bound in Theorem~\ref{th:ann}}
\label{S2.3} 

Let $\xi$ be as in {\bf (I)}, {\bf (IIa)}, {\bf (III)} or as in {\bf (IIb)} with $0<p<1/G(0)$.
By Propositions \ref{prop:wn}--\ref{prop:revss}, $\lambda_p(\cK)$ is a continuous, 
non-increasing and convex function of the components of $\cK$. Moreover,  
Propositions~\ref{prop:wn}--\ref{prop:revss} are still true when $\mathcal{K}=\kappa\in(0,\infty)$. 
It therefore follows that $\lambda_p(\cK) \leq \sup\{\lambda_p(\kappa)\colon\,\kappa\in\Supp(\cK)\}$. If $\xi$ is as in {\bf (IIb)} but with $p\geq 1/G(0)$, then by \cite[Theorem 1.4]{GdH06} the annealed Lyapunov exponents $\lambda_p(\kappa)$ are infinite for all $p\in\N$ and $\kappa\in[0,\infty)$. Hence, the upper bound in Theorem~\ref{th:ann} trivially holds in this case.

\subsection{Proof of Propositions~\ref{prop:wn}--\ref{prop:revss}}
\label{S2.4}

The proofs are, besides the proof of $\leq$ in~(\ref{eq:varISRW}), essentially 
straightforward extensions of the proofs of 
\cite[Lemma III.1.1]{CM94}, \cite[Proposition 2.1]{CGM12} and 
\cite[Proposition 2.2.2]{GdHM07} for $\cK \equiv \kappa \in(0,\infty)$.
We \emph{only indicate the main steps} (and so the arguments in this 
section are not self-contained). 
 
\subsubsection{Proof of Propositions~\ref{prop:wn},~\ref{prop:fnirw}~and~\ref{prop:revss}}
\label{S2.4.1}

\bpr
As mentioned in Section \ref{S2.1}, the Feynman-Kac formulas for the annealed 
Lyapunov exponents for white noise and finitely many independent random walks 
are similar, since the term 
\be{}
\sum_{1\leq i<j\leq p}
\int_0^{t} \one\{X_i^{\cK}(s)=X_j^{\cK}(s)\}\, ds
\ee
in (\ref{eq:pmwn}) for white noise is replaced by the term 
\be{}
\sum_{i=1}^{p}\sum_{j=1}^{n}
\int_0^{t} \one\{X_i^{\cK}(s)=X_j^{\rho}(s)\}\, ds.
\ee
in (\ref{eq:pmfnirw}) for finitely many independent random walks. Therefore a slight 
adaptation of the proof of Proposition \ref{prop:fnirw} below is enough to get the 
corresponding result for $\xi$ being space-time white noise, i.e., $\xi$ being as in {\bf(I)}.  

The proofs of Propositions~\ref{prop:fnirw},~and~\ref{prop:revss} follow the same line 
of argument as the proofs of \cite[Proposition 2.1]{CGM12} and 
\cite[Proposition 2.2.1]{GdHM07},  respectively, for $\cK \equiv \kappa$. Below we detail 
how to adapt the proofs. Consider the Markov process $Y=(Y(t))_{t \geq 0}$ with generator
\be{GKappa}
G_{V}^{\cK} = 
\begin{cases}
L_{1} + \sum_{i=1}^p \Delta_{i}^{\cK} + V_{1} \quad \text{ on } \ell^2(m^n\otimes m^p),
&\text{ if } \xi \text{ is as in {\bf(IIa)}},\\
L_{2} + \sum_{i=1}^p \Delta_{i}^{\cK} + V_{2} \quad \text{ on } L^2(\mu\otimes m^p),
&\text{ if } \xi \text{ is as in {\bf(III)}},
\end{cases}
\ee
where $L_{1}$ and $L_{2}$ are the generators of {\bf(IIa)} and {\bf(III)} respectively, 
$\Delta_i^{\cK}$ is given as in \eqref{Lpdef} but acting on the second coordinate of $f\in\ell^2(m^n\otimes m^p)$ and $f\in L^2(\mu\otimes m^p)$ (if $\xi$ is as in {\bf(IIa)} and {\bf(III)} respectively), and $V_{1}$ (as in \cite[Eq. (16)]{CGM12})
and $V_{2}$ (as in \cite[Eq. (2.2.2)]{GdHM07}) by 
\be{V1def}
V_{1}(x,y) = \sum_{i=1}^{n}\sum_{j=1}^{p}\delta_{0}(x_{j}-y_{i}),
\quad x=(x_1,\cdots,x_p)\in\Z^{dp},\, y=(y_1,\cdots,y_n)\in\Z^{dn},
\ee
and
\be{V2def}
V_{2}(\eta,x) = \sum_{i=1}^{p}\eta(x_i),
\quad \eta\in\Omega,\, x=(x_1,\cdots,x_p)\in\Z^{dp}.
\ee
Since $L_{1}$ and $L_{2}$ are reversible and bounded, and $\cK$ has compact support 
and is symmetric, $G_{V}^{\cK}$ is a \emph{bounded self-adjoint} operator. 

\medskip\noindent
{\bf Upper bound:}
Let
\be{kappainfsup}
\kappa^\ast = \esssup[\mathrm{supp}(\cK)],
\qquad \kappa_{\ast} = \essinf[\mathrm{supp}(\cK)],
\ee 
and let $B_R(t) \subset \Z^d$ be the box of radius $R(t) = t\log t$ centered at the origin. Then, for any fixed realization of $\cK$, we have 
\be{Ptrunc*}
P_0\big(X^{\cK}(1)\notin B_{R(t)}\big)
\leq P\big(N(2d\kappa^{\ast})\geq R(t)\big)
\leq\exp[-C(d,\kappa^\ast)R(t)]
\ee
for some $C(d,\kappa^\ast)>0$, where $N(2d\kappa^{\ast})$ is Poisson distributed with 
parameter $2d\kappa^*$. Thus, $\lim_{t\to\infty}\frac{1}{t} \log P_0(X^{\cK}(1)\notin 
B_{R(t)})=-\infty$.

\medskip\noindent
{\bf Lower bound:}
Since $\cK$ is bounded away from zero and infinity, it follows that for any finite
$K\subset\Z^d$ there exists $C>0$ such that
\be{}
P_0\big(X^{\cK}(1)=x\big)\geq\bigg(\frac{\kappa_{\ast}}
{2d\kappa^{\ast}}\bigg)^{\|x\|}e^{-2d\kappa^{\ast}}\,\frac{(2d\kappa_{\ast})^{\|x\|}}{\|x\|!}\geq C.
\qquad \forall\,x\in K. 
\ee
Picking $K=K_{\delta}$, $\delta>0$, we get, as in \cite[Eq.\ (2.2.10)]{GdHM07}, 
\be{}
\begin{aligned}
E_0^{\otimes p}\otimes E_0^{\otimes n}&\left(\exp\left\{\int_0^t V_1(Y(s))\, ds\right\}\right)\\
&\geq (C_{\delta}^{\mathcal{K}})^p(C_{\delta}^{\rho})^n
\sum_{\substack{x_1,\ldots, x_p\in K_{\delta}\\y_1,\ldots, y_n\in K_{\delta}}}
E_{x_1,\ldots, x_p}\otimes E_{y_1,\ldots, y_n}\left(\exp\left\{\int_0^{t-1}V_1(Y(s))\, ds\right\}\right),
\end{aligned}
\ee
if $\xi$ is as in {\bf (IIa)}, and
\be{}
\E_{\mu,0,\ldots,0}\left(\exp\left\{\int_{0}^tV_2(Y(s))\, ds\right\}\right)
\geq (C_{\delta}^{\cK})^p\sum_{x_{1},\ldots,x_{p}\in K_{\delta}}
\E_{\mu,x_{1},\ldots,x_{p}}\left(\exp\left\{\int_{0}^{t-1}V_2(Y(s))\, ds\right\}\right),
\ee
if $\xi$ is as in {\bf(III)}.
Here $C_{\delta}^{\mathcal{K}}=\min_{x\in K_{\delta}}P(X^{\cK}(1)=x)>0$ and
$C_{\delta}^{\rho}= \min_{x\in K_{\delta}}P_0(X^{\rho}(1)=x)>0$.  Now proceed as in the proof of 
\cite[Proposition 2.2.1]{GdHM07} and then apply the Rayleigh-Ritz formula as in the proof of 
\cite[Proposition 2.2.2]{GdHM07}.
\epr

\subsubsection{Proof of Proposition~\ref{prop:infirw}}
\label{S2.4.3}

\bpr
We only prove the case $p=1$, the extension to general $p$ being straightforward.
The proof of Proposition \ref{prop:infirw} is divided into 2 Steps.

\medskip\noindent
{\bf Step 1:} 
We first show that $\lambda_1(\cK)$ is bounded from above by the right-hand 
side of \eqref{eq:varISRW}. Recall \eqref{VNdef}.

\begin{claim}
\label{cl:upper}
There is a sequence of constants $C_t$, $t> 0$, with $\lim_{t\to\infty}C_t=\infty$ such that for all $N\in\N$ and $t>0$,
\begin{equation}
\label{eq:sharpupper}
\EE_{\mu,0} \Bigg(\exp\Bigg\{\int_0^t V_N\big(\xi_s,X^{\cK}(s)\big)\, ds\Bigg\}
\Bigg) \leq e^{t\lambda(V_N)}(2t\log t+1)^{d} + e^{-C_t t},
\end{equation}
where $\EE_{\mu,0}$ denotes expectation w.r.t.\ the joint process $(\xi,X^{\cK})$ when 
$\xi$ is drawn from $\mu$ and $X^{\cK}$ starts at $0$, and 
\begin{equation}
\label{eq:lambdaVN}
\lambda(V_N) = \sup_{\|f\|_{L^2\big(\N_0^{\Z^d}\times\Z^{d},\mu\otimes m\big)}=1}
\left\langle \left(L+\Delta^{\cK}+V_N\right)f,f\right\rangle.
\end{equation}
\end{claim}

\noindent
Claim~\ref{cl:upper} implies the upper bound in Proposition~\ref{prop:infirw}. Indeed, 
via monotone convergence, for all $t>0$,
\begin{equation}
\label{eq:monotone}
\begin{aligned}
\EE_{\mu,0}
&\Bigg(\exp\Bigg\{\int_0^t\xi(X^{\cK}(s),s)\, ds\Bigg\}\Bigg)\\
&= \lim_{N\to\infty} \EE_{\mu,0}\Bigg(\exp\Bigg\{\int_0^t 
V_N\big(\xi_s,X^{\cK}(s)\big)\, ds\Bigg\}\Bigg)\\
&\leq \sup_{N\in\N} \{e^{t\lambda(V_N)}(2t\log t+1)^{d} +e^{-C_t t}\} 
= e^{t\sup_{N\in\N}\lambda(V_N)}(2t\log t +1)^{d} + e^{-C_t t}.
\end{aligned}
\end{equation}
Taking the logarithm, dividing by $t$ and letting $t\to\infty$,leads to the desired upper bound.

Before we begin the proof of Claim~\ref{cl:upper} we recall some facts from G\"artner 
and den Hollander~\cite{GdH06}. A slight generalization of \cite[Proposition 2.1]{GdH06} 
states that
\begin{equation}
\label{eq:ISRWrep}
\begin{aligned}
&\EE_{\mu,0}\Bigg(\exp\Bigg\{\int_0^t\xi(X^{\cK}(s),s)\, ds\Bigg\}u_0(X^{\cK}(t))\Bigg) \\
&\qquad 
= e^{\nu t}E_{0}\Bigg(\exp\Bigg\{\nu\int_0^{t} w(X^{\cK}(s),s)\, ds\Bigg\}u_0(X^{\cK}(t))\Bigg).
\end{aligned}
\end{equation}
Here, the function $w$ is the 
solution of the equation
\begin{equation}
\label{eq:w}
\left\{
\begin{array}{l}
\displaystyle\frac{\partial}{\partial t}w(x,t) 
= \Delta w(x,t) + \Bigg[\sum_{i=1}^{p}\delta_{X_i^{\cK}(t)}(x)\Bigg]\{w(x,t)+1\},\bigg.\\
w(x,0) = 0,
\end{array}
\qquad x\in\Z^d,\,t\geq 0\, .
\right.
\end{equation}
Moreover, \cite[Propositions 2.2--2.3]{GdH06} state that there is a function $\bar{w}\colon\,\Z^d\times
[0,\infty)\to\R$ such that: (i) $w(x,t) \leq \bar{w}(0,t)$ for all $x\in\Z^d, t\geq 0$; (ii) $t\mapsto \bar{w}(0,t)$ 
is non-decreasing with limit
\begin{equation}
\label{eq:wlimit}
\bar{w}(0) =
\left\{
\begin{array}{ll}
\frac{pG(0)}{1-pG(0)}, &\mbox{if } 0<p<1/G(0),\\
\infty, &\mbox{otherwise.}
 \end{array}
 \right.
\end{equation}

We are now ready to prove Claim~\ref{cl:upper}. We use ideas from Kipnis and 
Landim~\cite[Appendix 1.7]{KL99}.
Recall the uniform ellipticity assumption~\eqref{staterg} on the $\cK$-field.
Thus, by standard large deviation estimates of the number of jumps of $X^{\cK}$ 
and by \eqref{eq:ISRWrep}--\eqref{eq:wlimit}, there is a sequence of constants 
$C_t$ as in the statement of Claim~\ref{cl:upper} such that for all $t>0$ and $N\in\N$,
\begin{equation}
\label{eq:negligiblebox}
\EE_{\mu,0} \Bigg(\exp\Bigg\{\int_0^t V_N\big(\xi_s,X^{\cK}(s)\big)\, ds\Bigg\}
\one\{X^{\cK}([0,t])\subsetneq B_{R(t)}\}\Bigg) \leq e^{-C_t t}.
\end{equation} 
Here, $B_{R(t)}$ denotes the box centered at the origin with side length $R(t)=t\log t$.
We now make use of the following fact (which follows from Demuth and van 
Casteren~\cite[Theorem 2.2.5]{DC00}). Let $W\colon\,\N_0^{\Z^d}\times\Z^d \to \R$ be a bounded 
function. Then $L + \Delta^{\cK} + W$ is a self-adjoint operator on $L^2(\N_0^{\Z^d} \times \Z^d,
\mu\otimes m)$, and is the generator of the semigroup
\begin{equation}
\label{eq:semigr}
(P_t^{W}f)(\eta,x) = \EE_{\eta,x}\Bigg(\exp\Bigg\{\int_0^t W\big(\xi_s,X^{\cK}(s)\big)\, ds\Bigg\} 
f(\xi_t,X^{\cK}(t))\Bigg), \qquad t>0.
\end{equation}
In particular, the function 
$v_t(\eta,x) = (P_t^{V_N}\bar{f})(\eta,x)$ with $\bar{f}(\eta,x) = \one\{x\in B_{R(t)}\}$
is a solution of the equation
\begin{equation}
\label{eq:semigroupeq}
\left\{\begin{array}{ll}
&\frac{\partial}{\partial t} v_t(\eta,x)
= (L+\Delta^{\cK} + V_N)v_t(\eta,x),\\[0.2cm]
&v_0(\eta,x) = \bar{f}(\eta,x),
\end{array}\right.
\quad \eta\in\N^{\Z^d},\ x\in\Z^d,\ t\geq 0.
\end{equation}
Here $V_N$ acts as a multiplication operator.
Since $\bar{f}\in L^2(\N_0^{\Z^d}\times\Z^d, \mu\otimes m)$ we can write
\begin{equation}
\label{eq:semigrouprewrite}
\begin{aligned}
\EE_{\mu,0}
&\Bigg(\exp\Bigg\{\int_0^t V_N\big(\xi_s,X^{\cK}(s)\big)\, ds\Bigg\}\one\{X^{\cK}([0,t])\subset B_{R(t)}\}\Bigg)\\
&\leq \EE_{\mu,0}
\Bigg(\exp\Bigg\{\int_0^t V_N\big(\xi_s,X^{\cK}(s)\big)\, ds\Bigg\}\one\{X^{\cK}(t)\in B_{R(t)}\}\Bigg)\\
&\leq  \int_{\N_0^{\Z^d}}\sum_{x\in\Z^d} \one\{x\in B_{R(t)}\}\, 
\EE_{\eta,x}\Bigg(\exp\Bigg\{\int_0^t V_N\big(\xi_s,X^{\cK}(s)\big)\, ds\Bigg\}
\one\{X^{\cK}(t)\in B_{R(t)}\}\Bigg)\, d\mu(\eta)\\
&= \langle P_t^{V_N}\bar{f},\bar{f}\rangle.
\end{aligned}
\end{equation}
Moreover, by \eqref{eq:semigroupeq}, for all $t>0$,
\begin{equation}
\label{eq:derivative}
\begin{aligned}
\frac{\partial}{\partial t} \|P_t^{V_N}\bar{f}\|_{L^2(\N_0^{\Z^d}\times\Z^d, \mu\otimes m)}^2
&= \int_{\N_0^{\Z^d}}\sum_{x\in\Z^d} \left[2\left(L+\Delta^{\cK}+V_N\right)(P_t^{V_N}\bar{f})(\eta,x)
\times (P_t^{V_N}\bar{f})(\eta,x)\right]\, d\mu(\eta)\\
&= 2\left\langle \left(L+\Delta^{\cK}+V_N\right)P_t^{V_N}\bar{f},P_t^{V_N}\bar{f}\right\rangle\\
&\leq 2\lambda(V_N)\|P_t^{V_N}\bar{f}\|_{L^2(\N_0^{\Z^d}\times\Z^d, \mu\otimes m)}^2,
\end{aligned}
\end{equation}
where interchanging the derivative and the scalar product is justified by dominated 
convergence in combination with Lemma~\ref{lem:B} in the appendix section. Further 
note that 
\begin{equation}
\|P_0^{V_N}
\bar{f}\|_{L^2(\N_0^{\Z^d}\times\Z^d, \mu\otimes m)}^{2}= |B_{R(t)}|\leq (2t\log t+1)^d,
\end{equation}
so that, by Gronwall's lemma,
\begin{equation}
\label{eq:Gronwall}
\|P_t^{V_N}\bar{f}\|_{L^2(\N_0^{\Z^d}\times\Z^d, \mu\otimes m)}^{2}
\leq e^{2\lambda(V_N)t}(2t\log t+1)^d.
\end{equation}
Using Cauchy-Schwarz and $\|\bar{f}\|_{L^2(\N_0^{\Z^d}\times\Z^d, \mu\otimes m)}^{2}=1$, 
we obtain that
\begin{equation}
\label{eq:CS}
\left\langle P_t^{V_N}\bar{f},\bar{f}\right\rangle \leq e^{\lambda(V_N)t} (2t\log t +1)^{d}.
\end{equation}
The claim follows by combining \eqref{eq:negligiblebox}, \eqref{eq:semigrouprewrite} 
and \eqref{eq:CS}.

\medskip\noindent
{\bf Step 2:} 
It remains to show that $\lambda_1(\cK)$ is bounded from below by the right-hand 
side of \eqref{eq:varISRW}. The proof follows the same line of argument as the 
proof of \cite[Proposition 2.2.1]{GdHM07} for $\cK \equiv \kappa$. The details to 
adapt it are left to the reader since they are similar to those given in the proof of 
the lower bound in Section~\ref{S2.4.1}.
\epr


\section{Annealed Lyapunov exponents: confinement approximation and lower bound in Theorem~\ref{th:ann}}
\label{S3}

In Section~\ref{S3.1} we show that the annealed Lyapunov exponents for $\cK\equiv\kappa$ 
do not change when the random walk in the Feynman-Kac formula \eqref{FK} is confined to 
a slowly growing box (Proposition~\ref{prop:clusstrat}). In Section~\ref{S3.2} we use this
result to prove the lower bound in Theorem~\ref{th:ann}, i.e.,
$\sup\{\lambda_p(\kappa):\, \kappa\in\mathrm{Supp}(\mathcal{K})\}
\leq \lambda_p(\mathcal{K})$.
Throughout this section we assume that $u_0=\delta_0$, see Proposition~\ref{prop:locin} for a justification of that assumption.

\subsection{Confinement approximation}
\label{S3.1}

\begin{proposition}
\label{prop:clusstrat}
Fix $p\in\N$ and $\kappa>0$, and let $\xi$ be as in {\rm {\bf (I)--(III)}}. Fix a non-decreasing 
function $L\colon\,[0,\infty)\to[0,\infty)$ such that $\lim_{t\to\infty} L(t)=\infty$. Then
\be{BoxALE}
\lim_{t\to\infty} \frac{1}{pt} \log \EE\Bigg[E_{0}\left(\exp\left\{\int_{0}^t 
\xi(X^\kappa(s),s)ds\right\}\delta_0(X^{\kappa}(t))\one\left\{X^\kappa[0,t]\subset B_{L(t)}(0)\right\}\right)^{p}\Bigg]
=\lambda_{p}(\kappa).
\ee
\end{proposition}

\begin{proof}
We write out the proof for the dynamics {\bf (I)}, namely for space-time white noise. Given $p$ 
independent simple random walks $X_1^{\kappa},X_2^{\kappa},\ldots, X_p^{\kappa}$, write 
$\bar{X}^{\kappa}= (X_1^{\kappa},X_2^{\kappa},\ldots, X_p^{\kappa})$. For $0\leq s<t<\infty$, 
define 
\begin{equation}
\label{eq:chi}
\begin{aligned}
\Xi^{\WN}(s,t)= E_{0}^{\otimes p}
\Bigg(\exp\Bigg\{&\sum_{1\leq i<j\leq p}\int_0^{t-s}
\one\{X_{i}^{\kappa}(v)=X_{j}^{\kappa}(v)\}\, dv\Bigg\}\\
&\times\delta_0(\bar{X}^{\kappa}(t-s))\,
\one\Big\{\bar{X}^{\kappa}[0,t-s]\subseteq B_{L(t-s)}(0)\Big\}\Bigg),
\end{aligned}
\end{equation}
where, with a slight abuse of notation, we redefine $B_{L(t)}(0)=[-L(t),L(t)]^{dp} \cap \Z^{dp}$. 
Pick $u\in[s,t]$. Using that $L$ is non-decreasing, inserting $\delta_0(\bar{X}^{\kappa}(u-s))$, 
and using the Markov property of $\bar{X}^{\kappa}$ at time $u-s$, we see that
\begin{equation}
\label{eq:superaddclus}
\Xi^{\WN}(s,t)\geq \Xi^{\WN}(s,u)\Xi^{\WN}(u,t).
\end{equation}
Hence,
\begin{equation}
\label{eq:limexists}
\lim_{t\to\infty}\frac1t\log\Xi^{\WN}(0,t) 
\end{equation}
exists. Thus, in order to prove Proposition~\ref{prop:clusstrat} it suffices to prove that 
\begin{equation}
\label{eq:discretelim}
\lim_{n\to\infty}\frac{1}{pnT}\log\Xi^{\WN}(0,nT)= \lambda_p(\kappa),
\quad T\in(0,\infty).
\end{equation}

Fix $T>0$. First, inserting $\one\{\bar{X}^{\kappa}[0,nT]\subseteq B_{L(nT)}(0)\}$ and second inserting 
$\delta_0(\bar{X}^{\kappa}(kT))$, $k\in\{1,2,\ldots, n-1\}$, and using the Markov 
property of $\bar {X}^{\kappa}$ at times $kT$ for the same set of indices, we get
\begin{equation}
\label{eq:regenerationarg}
\begin{aligned}
&E_0^{\otimes p}\Bigg(\exp\Bigg\{\sum_{1\leq i<j\leq p}\int_0^{nT}
\one\{X_{i}^{\kappa}(v)=X_{j}^{\kappa}(v)\}\, 
dv\Bigg\}\delta_0(\bar{X}^{\kappa}(nT))\Bigg)\\
&\geq \Xi^{\WN}(0,nT)\\
&\geq \prod_{k=1}^{n} E_0^{\otimes p}
\Bigg(\exp\Bigg\{\sum_{1\leq i<j\leq p}\int_0^{T}
\one\{X_{i}^{\kappa}(v)=X_{j}^{\kappa}(v)\}\, 
dv\Bigg\}\delta_0(\bar{X}^{\kappa}(T))\\
&\qquad\qquad\qquad \times \one\big\{\bar{X}^{\kappa}[0,T]\subseteq B_{L(nT)}(0)\big\}\Bigg).
\end{aligned}
\end{equation}
Taking the logarithm, dividing by $pnT$, and letting $n\to\infty$ followed by $T\to\infty$, we obtain
\begin{equation}
\label{eq:sandwicharg}
\lambda_p(\kappa)\geq \lim_{T\to\infty}\lim_{n\to\infty}\frac{1}{pnT} \log \Xi^{\WN}(0,nT)
\geq \lambda_p(\kappa),
\end{equation}
which is the desired claim. 

The proof for {\bf (II)--(III)} works along the same lines. To use the superadditivity argument as 
in \eqref{eq:superaddclus} and to get the inequalities in \eqref{eq:regenerationarg}, the same 
techniques as in the first step of the proof of Proposition~\ref{prop:locin} in Appendix \ref{App} 
may be applied.
\end{proof}

\subsection{Proof of the lower bound in Theorem~\ref{th:ann}}
\label{S3.2}

We give the proof for {\bf (I)}. The idea of the proof is to restrict the random walk to a box that 
slowly increases with time such that the $\cK$-field is constant on this box. The existence of 
such a box is guaranteed by the clustering property of $\cK$ stated in Definition~\ref{def:clustering}. 
Proposition~\ref{prop:clusstrat} then yields that the resulting Lyapunov exponent equals 
$\lambda_p(\kappa)$ with $\kappa$ the value of $\cK$ on this box.

\begin{proof}
The proof comes in 2 Steps.

\medskip\noindent
{\bf Step 1:} 
We first prove the lower bound in Theorem~\ref{th:ann} under the assumption that $\Supp(\cK)
= \{\kappa_1,\kappa_2\}$, $0<\kappa_1<\kappa_2<\infty$. By the clustering property of $\cK$, 
there is a function $L\colon\, [0,\infty) \to [0,\infty)$ with $\lim_{t\to\infty} L(t) = \infty$ such that 
there is a $x(\kappa_l,t)\in\Z^d$ with $g_l(t) \overset{def}{=} \|x(\kappa_l,t)\| \in o(t)$ such that 
$\cK(x,y)= \kappa_l$ for all edges $(x,y) \in B_{L(t)}(x(\kappa_l,t))$, $l\in\{1,2\}$. We fix $l\in\{1,2\}$ 
and, as in the proof of Proposition~\ref{prop:clusstrat}, denote by $\bar{X}^{\cK}$ the 
$\Z^{dp}$-valued process $(X_{1}^{\cK},\ldots, X_{p}^{\cK})$. An application of the Markov 
property of $\bar{X}^{\cK}$ at times $g_l(t)$ and $t-g_l(t)$ yields
\begin{equation}
\label{eq:annlowerest}
\begin{aligned}
E_0^{\otimes}
&\left(\exp\left\{\sum_{1\leq i<j\leq p}\int_0^t
\one\{X_{i}^{\cK}(s)= X_{j}^{\cK}(s)\}\, ds\right\}\delta_0(\bar{X}^{\cK}(t))\right)\\
&\geq E_0^{\otimes p}\left(\exp\left\{\sum_{1\leq i<j\leq p}\int_{0}^{g_l(t)}
\one\{X_{i}^{\cK}(s)=X_{j}^{\cK}(s)\}\, ds\right\}\delta_{x(\kappa_l,t)}\big(\bar{X}^{\cK}(g_l(t))\big)\right)\\
&\times E_{x(\kappa_l,t)}^{\otimes p}\left(\exp\left\{\sum_{1\leq i<j\leq p}\int_{0}^{t-2g_l(t)}
\one\{X_{i}^{\cK}(s)= X_{j}^{\cK}(s)\}\, ds\right\}\delta_{x(\kappa_l,t)}\big(\bar{X}^{\cK}(t-2g_l(t))\big)\right)\\
&\times  E_{x(\kappa_l,t)}^{\otimes p}\left(\exp\left\{\sum_{1\leq i <j\leq p}\int_{0}^{g_l(t)}
\one\{X_{i}^{\cK}(s)= X_{j}^{\cK}(s)\}\, ds\right\}\delta_{0}\big(\bar{X}^{\cK}(g_l(t))\big)\right)\\
&\overset{\hbox{def}}{=} U_1(t)\times U_2(t)\times U_3(t).
\end{aligned}
\end{equation} 
Note that 
\begin{equation}
\label{eq:annU1}
U_1(t)
\geq P_0\left(X^{\cK}(g_l(t))=x(\kappa_{l},t)\right),
\end{equation}
which is bounded from below by
\begin{equation}
\label{eq:trivlowerbound}
\left(\frac{\kappa_1}{2d\kappa_2}\right)^{g_l(t)}
e^{-2d\kappa_2 g_l(t)}\,\frac{(2d\kappa_1 g_l(t))^{g_l(t)}}{g_l(t)!},
\end{equation}
so that $\lim_{t\to\infty}\frac1t\log U_1(t)=0$.  The same reasoning shows that also $\lim_{t\to\infty}
\frac1t\log U_3(t)=0$. To control $U_2$, we use the lower bound
\begin{equation}
\label{eq:annconfine}
\begin{aligned}
&U_2(t)
\geq E_{x(\kappa_l,t)}^{\otimes p}\Bigg(\exp\left\{\sum_{1\leq i<j\leq p}\int_{0}^{t-2g_l(t)} 
\one\{X_{i}^{\cK}(s)= X_{j}^{\cK}(s)\}\,ds\right\}
\delta_{x(\kappa_l,t)}\big(\bar{X}^{\cK}(t-2g_l(t))\big)\\
&\qquad\qquad\qquad\qquad\times
\one\Big\{\bar{X}^{\cK}[0,t-2g_l(t)]\subseteq B_{L(t)-1}(x(\kappa_l,t))\Big\}\Bigg).
\end{aligned}
\end{equation}
Note that $X^{\cK}$ on the event 
$\{X^{\cK}[0,t]\subseteq B_{L(t)-1}(x(\kappa_l,t))\}$ is distributed as a random walk with diffusion 
constant $\kappa_l$ confined to stay in this  box. Hence, by the shift invariance of $\bar{X}^{\kappa}$ 
in space and Proposition~\ref{prop:clusstrat},
\begin{equation}
\label{eq:annU2}
\begin{aligned}
U_2(t)
&\geq E_{0}^{\otimes p}\Bigg(\exp\Bigg\{\sum_{1\leq i<j\leq p}\int_{0}^{t-2g_l(t)} 
\one\{X_{i}^{\kappa_l}(s)= X_{j}^{\kappa_l}(s)\}\, ds\Bigg\}
\delta_{0}\big(\bar{X}^{\kappa_l}(t-2g_l(t))\big)\\
&\qquad\qquad\qquad\qquad\times
\one\Big\{\bar{X}^{\kappa_l}[0,t-2g_l(t)]\subseteq B_{L(t)-1}(0)\Big\}\Bigg)\\
&\geq e^{\lambda_p(\kappa_l)(t-2g_l(t))p + o(t)}.
\end{aligned}
\end{equation}
Finally, (\ref{eq:annlowerest}--\ref{eq:annU2}) 
yield that
\begin{equation}
\label{eq:annlowerdiscrete}
\lambda_p(\cK) \geq \max\{\lambda_p(\kappa_1),\lambda_p(\kappa_2)\},
\end{equation}
which settles Theorem~\ref{th:ann} for the case where $\Supp(\cK)= \{\kappa_1,\kappa_2\}$, 
$\kappa_1,\kappa_2\in (0,\infty)$.

\medskip\noindent 
{\bf Step 2:} 
We next prove Theorem~\ref{th:ann} for the general case by reducing it to the setting of Step 1. 
Recall \eqref{kappainfsup}. Fix $n\in\N$. Given a realization of $\cK$, we define a discretization 
$\cK_n$ of $\cK$ by putting, for each $x,y \in \Z^d$,
\begin{equation}
\label{eq:annkappan}
\begin{aligned}
&\cK_n(x,y)\\
&= \left\{
\begin{array}{ll}
\kappa_{\ast}+(j-1)\frac{(\kappa^{\ast}-\kappa_{\ast})}{n}, 
&\mbox{if } \kappa_{\ast}+(j-1)\frac{(\kappa^{\ast}-\kappa_{\ast})}{n} \leq \cK(x,y) 
< \kappa_{\ast}+j\frac{(\kappa^{\ast}-\kappa_{\ast})}{n},\, 1 \leq j \leq n,\\
\kappa^{\ast}, &\mbox{if } \cK(x,y)=\kappa^{\ast}.
\end{array}
\right.
\end{aligned}
\end{equation}
A slight adaptation of Step 1 yields
\begin{equation}
\label{eq:annlambdan}
\lambda_p(\cK_n) \geq \max\{\lambda_p(\kappa),\, 
\kappa\in \Supp(\cK_n)\setminus\{\kappa^{\ast}\}\}.
\end{equation}
Here, the restriction to the set $\Supp(\cK_n)\setminus\{\kappa^{\ast}\}$ comes from the fact 
that $\tilde{\PP}(\cK(x,y)=\kappa^{\ast})=0$ is possible, e.g.\ when the distribution of $\cK$ 
is continuous. By Carmona and Molchanov~\cite[Proposition III.2.7]{CM94}, $\kappa\mapsto
\lambda_p(\kappa)$ is continuous, hence the right-hand side of \eqref{eq:annlambdan} 
converges to $\sup\{\lambda_p(\kappa),\, \kappa\in \Supp(\cK)\}$ as $n\to\infty$. Hence it 
suffices to show that $\limsup_{n\to\infty}\lambda_p(\cK_n)\leq \lambda_p(\cK)$.

To do so we borrow ideas from the proof of \cite[Theorem 1.2(i)]{GdHM11}. First we introduce 
the notation $\tilde{\cK}(x)= \sum_{y\in\Z^d}\cK(x,y)$, $x\in\Z^d$, and we define $\tilde{\cK}_n$ 
in a similar fashion. An application of Girsanov's formula yields 
that (see K\"onig, Salvi and Wolff~\cite[Lemma 2.1]{KSW12})
\begin{equation}
\label{eq:annGirsanov}
\begin{aligned}
E_{0}^{\otimes p}&\left(\exp\left\{\sum_{1\leq i<j\leq p}\int_0^t 
\one\{X_i^{\cK_n}(s)=X_j^{\cK_n}(s)\}\,ds\right\}\,
\delta_0(\bar{X}^{\cK_n}(t))\right)\\
&= E_0^{\otimes p}\bigg(\exp\left\{\sum_{1\leq i<j\leq p}\int_0^t 
\one\{X_i^{\cK}(s)=X_j^{\cK}(s)\}\,ds\right\}\,
\delta_0(\bar{X}^\cK(t))\\
&\times\exp\bigg\{\sum_{1\leq i\leq p}\sum_{l=1}^{N(X_i^{\cK};t)}
\log\Big[\frac{\cK_n(X_i^{\cK}(S_{l-1}),X_i^{\cK}(S_{l}))}{\cK(X_i^{\cK}(S_{l-1}),
X_i^{\cK}(S_{l}))}\Big]\\
&\qquad\qquad -\int_{0}^{t} \big[\tilde{\cK}_n(X_i^{\cK}(s))
-\tilde{\cK}(X_i^{\cK}(s))\big]\, ds\bigg\}\bigg),
\end{aligned}
\end{equation}
where $N(X^{\cK};t)$ denotes the number of jumps of the random walk $X^{\cK}$ with generator 
$\Delta^{\cK}$ up to time $t$. Note that $\frac{\cK_n(x,y)}{\cK(x,y)}\leq 1$ for all $x\sim y\in\Z^d$ 
and that $-\int_{0}^{t} [\tilde{\cK}_n(X^{\cK}(s))-\tilde{\cK}(X^{\cK}(s))]\, ds\leq 2dt/n$. Hence, the 
right-hand side of \eqref{eq:annGirsanov} is bounded from above by
\begin{equation}
\label{eq:annupperboundGir}
\begin{aligned}
E_0^{\otimes p}\Bigg(\exp\Bigg\{\sum_{1\leq i<j \leq p}
&\int_0^t \one\{X_i^{\cK}(s)=X_j^{\cK}(s)\}\,ds\Bigg\}\,
\delta_0(\bar{X}^{\cK}(t))\Bigg)e^{2dt/n}.
\end{aligned}
\end{equation}
Consequently, \eqref{eq:annGirsanov} and \eqref{eq:annupperboundGir} show that 
$\limsup_{n\to\infty}\lambda_p(\cK_n)\leq 
\lambda_p(\cK)$. This finishes the proof.
The proof for {\bf (II)} and {\bf (III)} is the same as above, with the additional restriction that 
$0<p<1/G(0)$ for {\bf (IIb)}. To get the inequality in \eqref{eq:annlowerest} we use the 
techniques in the first step of the proof of Proposition~\ref{prop:locin} in Appendix \ref{App}. 
By Castell, G\"un and Maillard~\cite[Theorem 1.1(ii)]{CGM12} and G\"artner and 
den Hollander~\cite[Theorem 1.5]{GdH06}, $\kappa\mapsto \lambda_p(\kappa)$ is 
continuous for {\bf (II)}, which allows us to take the limit on the right-hand side of 
\eqref{eq:annlambdan}. The continuity of $\kappa\mapsto\lambda_p(\kappa)$ for {\bf (III)} 
follows from Proposition~\ref{prop:revss}, which still holds when $\kappa$ is deterministic. 
Indeed, the variational formula in Proposition~\ref{prop:revss} shows that $\kappa\mapsto 
\lambda_p(\kappa)$ is convex. Since $\xi$ is bounded for {\bf (III)}, so is $\kappa\mapsto 
\lambda_p(\kappa)$, which yields the desired continuity. To obtain the result for {\bf (IIb)} with 
$p\geq 1/G(0)$, for which $\lambda_p(\kappa)=\infty$ for all $\kappa\geq 0$, we note that
averaging $u(0,t)^p$ first with respect to the trajectories $Y_j^y$ present in the definition of $\xi$, 
then with respect to the Poisson field $(N_y)_{y\in\Z^d}$ and using standard Feynman-Kac identities,
an adaption of the proof of \cite[Proposition 2.1]{GdH06} yields the estimate
\begin{equation}
\label{eq:strongcat}
\begin{aligned}
\EE[u(0,t)^{p}]
&\geq \EE\Bigg[E_0\Bigg(\exp\Bigg\{\int_0^{t}\xi(X^{\cK}(s),t-s)\, ds\Bigg\}
\one\Big\{X^{\cK}(s)=0\, \mbox{for all }s\in[0,t]\Big\}\Bigg)^p\Bigg]\\[0.2cm]
&\geq \exp\Bigg\{-pt\sum_{||x||=1}\cK(0,x)+p\nu t\Bigg\}
\exp\Bigg\{p\int_0^{t}\bar{w}(0,s)\, ds\Bigg\},
\end{aligned}
\end{equation}
where $\bar{w}$ solves the equation
\begin{equation}
\left\{
\begin{array}{ll}
\frac{\partial}{\partial t} \bar{w}(x,t)= \Delta \bar{w}(x,t) + \delta_0(x)[\bar{w}(x,t)+1],\\
w(x,0)= 0,
\end{array}
\quad x\in\Z^d,\, t\geq 0.
\right.
\end{equation}
To conclude it suffices to note that by \cite[Proposition 2.3]{GdH06} (with the notation $r_d
=1/G(0)$), $t\mapsto \bar{w}(0,t)$ is non-decreasing with $\lim_{t\to\infty}\bar{w}(0,t)=\infty$.
\end{proof}


\section{Quenched Lyapunov exponent: confinement approximation and lower bound}
\label{S4}

The proof of the existence of the quenched Lyapunov exponent follows along the lines 
of the proof of \cite[Theorem 1.1]{GdHM11}. In Section~\ref{S4.1} we show that a 
confinement approximation holds for $\cK\equiv\kappa$. In Section~\ref{S4.2} we use 
this result to prove Theorem~\ref{th:que}.

\subsection{Confinement approximation}
\label{S4.1}

\bp{BoxLemma}
Let $L\colon [0,\infty)\to[0,\infty)$ be non-decreasing with $\lim_{t\to\infty} L(t) = \infty$.
Then $\P$-a.s.\ and in $\P$-mean,
\be{BoxQLE}
\lim_{t\to\infty}\frac1t\log E_{0}\left(\exp\left\{\int_{0}^t \xi(X^\kappa(s),s)ds\right\}
\delta_{0}(X^\kappa(t))\,\one\Big\{X^\kappa[0,t]\subseteq B_{L(t)}(0)\Big\}\right)
= \lambda_{0}(\kappa).
\ee
\ep

\bpr
For $0\leq s\leq t<\infty$, define 
\be{chidef}
\Xi(s,t)= E_{0}\left(\exp\left\{\int_{0}^{t-s} \xi(X^\kappa(v),s+v)\, dv\right\}
\delta_{0}(X^\kappa(t-s))\,\one\Big\{X^\kappa[0,t-s]\subseteq B_{L(t-s)}(0)\Big\}\right).
\ee
Pick $u\in[s,t]$. Using that $L$ is non-decreasing and inserting 
$\delta_{0}(X^\kappa(u-s))$ under the expectation in (\ref{chidef}), we obtain
\be{chiineq}
\begin{aligned}
&\Xi(s,t)\geq
E_{0}\bigg(\exp\left\{\int_{0}^{u-s} \xi(X^\kappa(v),s+v)\, dv\right\}
\delta_{0}(X^\kappa(u-s))\,\one\Big\{X^\kappa[0,u-s]\subseteq B_{L(u-s)}(0)\Big\}\\
&\qquad\quad\quad \times\exp\left\{\int_{u-s}^{t-s} \xi(X^\kappa(v),s+v)\, dv\right\}
\delta_{0}(X^\kappa(t-s))\,\one\Big\{X^\kappa[u-s,t-s]\subseteq B_{L(t-u)}(0)\Big\}
\bigg).
\end{aligned}
\ee
Applying the Markov property of $X^\kappa$ at time $u-s$, we get
\be{}
\Xi(s,t)\geq\Xi(s,u)\Xi(u,t),
\quad 0\leq s\leq u\leq t<\infty.
\ee
Since $\xi$ is stationary and ergodic, and the law 
of $\{\Xi(u+s,u+t)\colon0\leq s\leq t<\infty\}$ is the same for all $u\geq 0$, it follows 
from Kingman's superadditive ergodic theorem that
\be{}
\label{eq:L1conv}
\lim_{t\to\infty}\frac1t\log\Xi(0,t)
\mbox{ exists } \P\mbox{-a.s.\ and in } \P\mbox{-mean, and is non-random.}
\ee
Thus, in order to prove (\ref{BoxQLE}), it suffices to show that
\be{}
\lim_{n\to\infty}\frac1{nT}\log \Xi(0,nT)=\lambda_{0}(\kappa),
\quad T\in(0,\infty).
\ee
Inserting $\one\{X^\kappa[0,nT]\subset B_{L(nT)}(0)\}$ and $\delta_0(X^{\kappa}(kT))$, 
$k\in\{1,2,\ldots, n-1\}$, and using the Markov property of $X^\kappa$ at times $kT$ for 
the same set of indices, we get
\be{Box1}
\begin{aligned}
&E_{0}\left(\exp\left\{\int_{0}^{nT} \xi(X^\kappa(s),s)ds\right\}
\delta_{0}(X^\kappa(nT))\right)\\
&\quad \geq \Xi(0,nT)\\
&\quad \geq \prod_{i=1}^{n}E_{0}\left(\exp\left\{\int_{0}^{T} \xi(X^\kappa(s),(i-1)T+s)ds\right\}
\delta_{0}(X^\kappa(T))\,\one\Big\{X^\kappa[0,T]\subseteq B_{L(nT)}(0)\Big\}\right).
\end{aligned}
\ee
Using that $\xi$ is invariant under time shifts, we get 
\be{Box3}
\begin{aligned}
&\frac1{nT}\,\EE\left[\log E_{0}\left(\exp\left\{\int_{0}^{nT} \xi(X^\kappa(s),s)ds\right\}
\delta_{0}(X^\kappa(nT))\right)\right]\\
&\quad \geq \frac1{nT}\,\EE\big[\log \Xi(0,nT)\big]\\
&\quad \geq\frac1T\,\EE\left[\log E_{0}\left(\exp\left\{\int_{0}^{T} \xi(X^\kappa(s),s)ds\right\}
\delta_{0}(X^\kappa(T))\,\one\Big\{X^\kappa[0,T]\subseteq B_{L(nT)}(0)\Big\}\right)\right].
\end{aligned}
\ee
Letting $n\to\infty$ followed by $T\to\infty$, and using the $L^1$-convergence in \eqref{eq:L1conv}, 
we arrive at the sandwich
\be{sandwich}
\lambda_{0}(\kappa)\geq \lim_{T\to\infty}\lim_{n\to\infty}\frac1{nT}\log \Xi(0,nT)
\geq \lambda_{0}(\kappa).
\ee
The convergence of the rightmost term in \eqref{Box3} to the rightmost term in \eqref{eq:L1conv} 
can be shown by a direct comparison between these two terms using  condition {\rm (4)} for $\xi$.
\epr

\subsection{Proof of Theorem~\ref{th:que}}
\label{S4.2}

With the help of Proposition~\ref{BoxLemma} we can now give the proof of Theorem~\ref{th:que}.

\begin{proof}
The proof comes in 2 Steps.

\medskip\noindent
{\bf Step 1:} 
We first prove Theorem~\ref{th:que} under the assumption $\Supp(\cK) = \{\kappa_1,\kappa_2\}$, 
$\kappa_1,\kappa_2\in (0,\infty)$. By the clustering property of $\cK$, there exists a function $L\colon 
[0,\infty) \to [0,\infty)$ with $\lim_{t\to\infty} L(t)$ $= \infty$ such that there is an $x(\kappa_i,t)\in\Z^d$ 
with $g_i(t) = \|x(\kappa_i,t)\| \in o(t)$ such that $\cK(x,y)= \kappa_i$ for all $(x,y)\in B_{L(t)}(x(\kappa_i,t))$, 
$i\in\{1,2\}$. We fix $i\in\{1,2\}$. An application of the Markov property of the random walk at times 
$g_i(t)$ and $t-g_i(t)$ yields
\begin{equation}
\label{eq:lowerest}
\begin{aligned}
E_0&\left(\exp\left\{\int_0^t\xi(X^{\cK}(s),s)\, ds\right\}\delta_0(X^{\cK}(t))\right)\\
&\geq E_0\left(\exp\left\{\int_{0}^{g_i(t)}\xi(X^{\cK}(s),s)\, ds\right\}
\delta_{x(\kappa_i,t)}\big(X^{\cK}(g_i(t))\big)\right)\\
&\times E_{x(\kappa_i,t)}\left(\exp\left\{\int_{0}^{t-2g_i(t)}\xi(X^{\cK}(s),s+g_i(t))\, ds\right\}
\delta_{x(\kappa_i,t)}\big(X^{\cK}(t-2g_i(t))\big)\right)\\
&\times  E_{x(\kappa_i,t)}\left(\exp\left\{\int_{0}^{g_i(t)}\xi(X^{\cK}(s),s+t-g_i(t))\, ds\right\}
\delta_{0}\big(X^{\cK}(g_i(t))\big)\right)\\
&\overset{\hbox{def}}{=} U_1(t)\times U_2(t)\times U_3(t).
\end{aligned}
\end{equation} 
Further note that, by Jensen's inequality,
\begin{equation}
\label{eq:U1}
\begin{aligned}
&\E[\log U_1(t)]\\
&\quad\geq \E\bigg[ E_0\bigg( \int_{0}^{g_i(t)}\xi(X^{\cK}(s),s)\, ds 
~\Big|~ X^{\cK}(g_i(t)) = x(\kappa_i,t)\bigg) \bigg]
+\log P_0\bigg(X^{\cK}(g_i(t))= x(\kappa_i,t)\bigg)\\
&\quad = E_0\bigg[ \int_{0}^{g_i(t)}\E\Big(\xi(X^{\cK}(s),s)\Big)\, ds
~\Big|~ X^{\cK}(g_i(t)) = x(\kappa_i,t)\bigg] +\log P_0\bigg(X^{\cK}(g_i(t))= x(\kappa_i,t)\bigg)\\
&\quad= \E(\xi(0,0)) g_i(t)+\log P_0\bigg(X^{\cK}(g_i(t))= x(\kappa_i,t)\bigg),
\end{aligned}
\end{equation}
where the interchange of the expectations is justified because
\begin{equation}
\label{eq:exchange}
\E\bigg[ E_0\bigg( \int_{0}^{g_i(t)}|\xi(X^{\cK}(s),s)|\, ds ~\Big|~ 
X^{\cK}(g_i(t)) = x(\kappa_i,t)\bigg) \bigg] = \E(|\xi(0,0)|)g_i(t) <\infty.
\end{equation}
A similar computation yields the same lower bound for $\E[\log U_3(t)]$. Note that the lower 
bounds are sublinear in $t$. To control $U_2$, note that $X^{\cK}$ restricted to the event 
$\{X^{\cK}[0,t] \subset B_{L(t)-1}(x(\kappa_i,t))\}$ is distributed as a random walk with diffusion 
constant $\kappa_i$ confined to stay in this box. Hence
\begin{equation}
\label{eq:confine}
\begin{aligned}
U_2(t)
&\geq E_{x(\kappa_i,t)}\Bigg(\exp\left\{\int_{0}^{t-2g_i(t)} \xi(X^{\kappa_i}(s),s+g_i(t))\,ds\right\}
\delta_{x(\kappa_i,t)}\big(X^{\kappa_i}(t-2g_i(t))\big)\\
&\qquad\qquad\qquad\qquad\times\one\Big\{X^{\kappa_i}[0,t-2g_i(t)]
\subset B_{L(t)-1}(x(\kappa_i,t))\Big\}\Bigg),
\end{aligned}
\end{equation}
so that, by the space-time shift invariance of $\xi$ and Proposition~\ref{BoxLemma},
\begin{equation}
\label{eq:U2}
\begin{aligned}
\E\log U_2(t)
&\geq\E\log E_{0}\Bigg(\exp\left\{\int_{0}^{t-2g_i(t)} \xi(X^{\kappa_i}(s),s)ds\right\}
\delta_{0}\big(X^{\kappa_i}(t-2g_i(t))\big)\\
&\qquad\qquad \times \one\Big\{X^{\kappa_i}[0,t-2g_i(t)]\subset B_{L(t)-1}(0)\Big\}\Bigg)\\
&\geq e^{\lambda_0(\kappa_i)(t-2g_i(t)) + o(t)}.
\end{aligned}
\end{equation}
Since, by the first part of Theorem~\ref{th:que}, we have the representation
\begin{equation}
\lambda_0(\cK) = \lim_{t\to\infty}\frac{1}{t}\E(\log u(0,t)),
\end{equation} 
(\ref{eq:lowerest}--\ref{eq:U2}) yield
\begin{equation}
\label{eq:lowerdiscrete}
\lambda_0(\cK) \geq \max\{\lambda_0(\kappa_1),\lambda_0(\kappa_2)\},
\end{equation}
which settles the claim for the case $\Supp(\cK)= \{\kappa_1,\kappa_2\}$, $\kappa_1,\kappa_2
\in (0,\infty)$.

\medskip\noindent 
{\bf Step 2:} 
The strategy to extend the proof to the general case works similarly as in the second 
step of the proof of Theorem~\ref{th:ann} in Section \ref{S3.2}. However, since we do 
not know whether $\kappa\mapsto\lambda_0(\kappa)$ is continuous, some modifications 
are needed (see \cite[Theorem 1.2(i)]{GdHM11}, where conditions are provided under 
which the quenched Lyapunov exponent $\lambda_0(\kappa)$ is Lipschitz continuous 
outside any neighbourhood of zero). Fix $n\in\N$ and given a realisation of $\mathcal{K}$ 
define a discretization $\mathcal{K}_n$ of $\mathcal{K}$ as in the second step of the 
proof of Theorem~\ref{th:ann}. An adaptation of Step 1 yields
\begin{equation}
\label{eq:qulowerbddis}
\lambda_0(\mathcal{K}_n)\geq \max\{\lambda_0(\kappa),\, 
\kappa\in\mathrm{Supp}(\mathcal{K}_n)\setminus\{\kappa^{\ast}\}\}.
\end{equation}
To continue, we claim that $\kappa\mapsto\lambda_0(\kappa)$ is lower semi-continuous 
on $(0,\infty)$. Indeed, fix $t>0$ and $\kappa\in(0,\infty)$, as well as a sequence 
$(\kappa_n)_{n\in\N}$ such that $\kappa_n\to\kappa$ as $n\to\infty$. An application of 
Girsanov's formula yields that
\begin{equation}
\label{eq:quGirsanov}
\begin{aligned}
&u(0,t;\kappa_n)\\
&= E_0\left(\exp\left\{\int_0^t\xi(X^{\kappa_n}(s),s)\, ds\right\}\delta_0(X^{\kappa_n}(t))\right)\\
&= E_0\left(\exp\left\{\int_0^t\xi(X^{\kappa}(s),s)\, ds\right\}\delta_0(X^{\kappa}(t))
\exp\left\{N(X^{\kappa};t)\log\Big[\frac{\kappa_n}{\kappa}\Big]-2dt[\kappa_n-\kappa]\right\}\right).
\end{aligned}
\end{equation}
Hence, from Fatou's lemma we get that $\liminf_{n\to\infty} u(0,t;\kappa_n)\geq u(0,t;\kappa)$. 
This shows that $\kappa\mapsto u(0,t;\kappa)$ is lower semi-continuous for all $t>0$. 
Using that
\begin{equation}
\label{eq:supproperty}
\lambda_0(\kappa)=\sup_{t>0}\frac{1}{t}\log u(0,t;\kappa)
\end{equation}
(see the proof of \cite[Theorem 1.1]{GdHM11}), we get the claim by using that suprema 
of lower semi-continuous functions are lower semi-continuous.

To proceed, let $M=\sup\{\lambda_0(\kappa),\,\kappa\in\mathrm{Supp}(\mathcal{K})\}$. 
We claim that the liminf of the right-hand side of \eqref{eq:qulowerbddis} is bounded from
below by $M$. We distinguish between two cases. If $M=\infty$, then for each $R>0$ 
there is $\kappa_R\in\mathrm{Supp}(\mathcal{K})$ such that $\lambda_0(\kappa_R)\geq R$.
Since $\kappa\mapsto\lambda_0(\kappa)$ is lower semi-continuous, for any $\varepsilon >0$ 
there is a neighborhood $\mathcal{U}_R$ of $\kappa_R$ such that $\lambda_0(\kappa)
\geq \lambda_0(\kappa_R)-\varepsilon$ for all $\kappa\in\mathcal{U}_R$. Hence, for all 
$R\geq 0$ and $\varepsilon >0$, we obtain
\begin{equation}
\label{eq:liminf1}
\liminf_{n\to\infty}\max\{\lambda_0(\kappa),\, 
\kappa\in\mathrm{Supp}(\mathcal{K}_n)\setminus\{\kappa^{\ast}\}\}\geq R-\varepsilon.
\end{equation}
From this we get the claim by letting $R\to\infty$. The case $M<\infty$, may be treated 
similarly. It only remains to show that $\limsup_{n\to\infty}\lambda_0(\mathcal{K}_n)\leq 
\lambda_0(\mathcal{K})$. But this works verbatim as in the second step of the proof of 
Theorem~\ref{th:ann}. 
\end{proof}


\section{Quenched Lyapunov exponent: failure of upper bound}
\label{S5}

In this section we provide an example where the upper bound fails for
a decorated version of $\Z^d$, namely, we show that there is a choice
of $\cK$ for which
\be{queinequal}
\lambda_0(\cK) > \sup\{\lambda_{0}(\kappa)\colon \kappa\in \Supp(\cK)\}. 
\ee

Let $(V,\cE)$ denote the usual graph associated with $\Z^d$, i.e., $V=\Z^d$ 
and $\cE=\{e(x,y)\colon\, x,y\in V,\, x\sim y\}$ is the set of edges connecting 
nearest-neighbour vertices of $V$. We consider $(V^\star,\cE^\star)$, 
a \emph{decorated} version of $(V,\cE)$, where $V^\star=V$ but 
\begin{equation}
\label{decoredge}
\cE^\star = \big\{(e_{1}(x,y),e_{2}(x,y))\colon\, x,y \in V,\, x\sim y\big\},
\end{equation}
i.e., we draw two edges rather than one, say red and green, between every 
pair of nearest-neighbour vertices of $\Z^d$. 

Pick any $\cK$ on $\cE^\star$ that has the \emph{alternating cluster property}, 
i.e., there exist boxes $B_{L(t)}$, with $\lim_{t\to\infty} L(t)=\infty$, on which all 
red edges have value $\kappa_1$ and all green edges have value $\kappa_2$. 
For such $\cK$, by the confinement approximation of Proposition~\ref{BoxLemma}, 
we have
\begin{equation}
\label{decorLyap}
\lambda_{0}(\cK)\geq \lambda_{0}\left(\cK\equiv(\kappa_{1},\kappa_{2})^\cE\right)
=\lambda_{0}(\kappa_{1}+\kappa_{2}),
\end{equation}
where $(\kappa_1,\kappa_2)^{\cE}$ means that all red egdes take value $\kappa_1$ 
and all green edges take value $\kappa_2$. In \cite{GdHM11} we exhibited a class of 
dynamic random environments $\xi$ for which 
\be{}
\begin{array}{ll}
&\kappa\mapsto\lambda_0(\kappa) \text{ is continuous on } [0,\infty),\\ 
&\lambda_0(\kappa)>\E(\xi(0,0)) \,\,\forall \kappa\in(0,\infty),\\
&\lim_{\kappa\to\infty} \lambda_0(\kappa)=\lambda_0(0)=\E(\xi(0,0)).
\end{array}
\ee
In particular, $\kappa\mapsto\lambda_0(\kappa)$ is not monotone on $[0,\infty)$.
Hence there exist $\bar \kappa_1,\bar\kappa_2\in(0,\infty)$ such that
\begin{equation}
\label{unimodalineq}
\lambda_0\left(\frac{\bar \kappa_1+\bar\kappa_2}{2}\right)
>\max\{\lambda_0(\bar \kappa_1),\lambda_0(\bar \kappa_2)\}.
\end{equation}
Picking $\kappa_{1}=\bar \kappa_{1}/2$ and $\kappa_{2}=\bar \kappa_{2}/2$,
we get
\begin{equation}
\label{decorLyaplwbd}
\lambda_0\left(\cK\equiv(\kappa_1,\kappa_2)^\cE\right)
>\max\left\{\lambda_0\left(\cK\equiv(\kappa_1,\kappa_1)^\cE\right),
\lambda_0\left(\cK\equiv(\kappa_2,\kappa_2)^\cE\right)\right\}.
\end{equation}
Combining (\ref{decorLyap}) and (\ref{decorLyaplwbd}), we arrive at (\ref{queinequal}). 

The above counterexample does not apply to $\lambda_0(\cK)$ on $(V,\cE)$.
Nevertheless, since all previous theory developed for $\lambda_0(\cK)$ on 
$(V,\cE)$ carries over to $(V^\star,\cE^\star)$, the above example shows that 
there is little hope for the upper bound to hold for $\Z^d$.


\appendix


\section{Restriction to a localized initial condition}
\label{App}

In this appendix we prove Proposition \ref{prop:locin}. The proof is somewhat 
long and technical, but the flexibility in the choice of initial condition is important.
The proof is an adaptation of the proof of Drewitz, G\"artner, Ramirez and 
Sun~\cite[Theorem 4.1]{DGRS12}. Throughout this section we fix $p\in\N$.

\subsection{Dynamics {\bf (I)}} 

\noindent
\begin{proof}
Recall the representation of the $p$-th moment of $u(0,t)$ in \eqref{eq:pmwn}, and the notation
$\bar{X}^{\kappa}= (X_1^{\kappa},X_2^{\kappa},\ldots, X_p^{\kappa})$. For $0\leq s < t <\infty$ 
and $y,z\in\R^{dp}$ such that $ys, zt\in\Z^{dp}$, write
\begin{equation}
\label{eq:Xi}
\Xi^{\WN}_{y,z}(s,t) = E_{ys}^{\otimes p}\Bigg(\exp\Bigg\{
\sum_{1\leq i< j\leq p}\int_{0}^{t-s}\one\Big\{X_i^{\kappa}(v)=X_j^{\kappa}(v)\Big\}\, dv\Bigg\}
\one\{\bar{X}^{\kappa}(t-s)=zt\}\Bigg),
\end{equation}
where under $E_{ys}^{\otimes p}$ the process $\bar{X}^{\kappa}$ starts in $ys$. Abbreviate 
$\Xi^{\WN}_y(s,t)=\Xi^{\WN}_{y,y}(s,t)$. It is enough to show the existence of a concave and 
symmetric function $\alpha\colon\,\R^{dp}\to\R$ such that, for all compact $K\subset \R^{dp}$,
\begin{equation}
\label{eq:uniform}
\lim_{t\to\infty}\sup_{y\in Kt\cap \Z^{dp}}
\Big|\frac{1}{t}\log \Xi^{\WN}_{y/t}(0,t)-\alpha(y/t)\Big| =0.
\end{equation}
Indeed, suppose that such a function exists. A short computation shows that $\alpha$ obtains a 
global maximum at zero. Moreover, a standard large deviation estimate for the number of jumps 
of $\bar{X}^{\kappa}$ shows that there is a compact subset $K\subset\R^{dp}$ such that
\begin{equation}
\label{eq:negligible}
\limsup_{t\to\infty} \frac{1}{t}\log E_{0}^{\otimes p}\Bigg(\exp\Bigg\{
\sum_{1\leq i< j\leq p}\int_{0}^{t}\one\{X_i^{\kappa}(v)=X_j^{\kappa}(v)\}\, dv\Bigg\}
\one\{\bar{X}^{\kappa}([0,t])\subsetneq Kt\}\Bigg) \leq -1.
\end{equation}
Hence, given such a set $K$, it is enough to focus on the contribution coming from those 
random walk paths such that $\{\bar{X}^{\kappa}[0,t]\subseteq Kt\}$. Note that necessarily 
$0\in K$. Fix $\varepsilon>0$. By the approximation property of $\alpha$ in \eqref{eq:uniform} 
we can find a $t_0\geq 0$ such that, for all $t\geq t_0$,
\begin{equation}
\label{eq:compare}
\alpha(0)-\varepsilon \\
\leq \frac{1}{t}\log \sum_{ {y\in Kt} \atop  {yt\in\Z^{dp}} }\Xi^{\WN}_{y/t}(0,t) 
\leq \frac{1}{t}\log|Kt| + \alpha(0)+\varepsilon,
\end{equation}
which yields the desired claim.

The proof of the existence of $\alpha$ is divided into 3 Steps.

\medskip\noindent
{\bf Step 1:} We first show the existence of a function $\alpha\colon\,\Q^{dp}\to \R$ such that, 
for all $y\in\Q^{dp}$,
\begin{equation}
\label{eq:discretealpha}
\lim_{\substack{t\to\infty\\yt\in\Z^{dp}}}
\frac1t \log \Xi^{\WN}_y(0,t) = \alpha(y).
\end{equation} 
To that end, we fix $y\in\Q^{dp}$ and take $0\leq s <u< t$ such that $ys, yu, yt\in\Z^{dp}$. Forcing 
$\bar{X}^{\kappa}$ to be at position $yu$ at time $u-s$, an application of the Markov property of 
$\bar{X}^{\kappa}$ at time $u-s$ yields
\begin{equation}
\label{eq:superadd}
\Xi^{\WN}_y(s,t) \geq \Xi^{\WN}_y(s,u)\,\Xi^{\WN}_y(u,t).
\end{equation}
Consequently, $(s,t)\mapsto \log\Xi^{\WN}_y(s,t)$ is superadditive, and the claim in 
\eqref{eq:discretealpha} follows.

\medskip\noindent
{\bf Step 2:}
To extend $\alpha$ to a function on $\R^{dp}$ and to get uniform convergence on compacts as 
in \eqref{eq:uniform}, we show that for any compact subset $K\subset\R^{dp}$,
\begin{equation}
\label{eq:cont}
\lim_{\varepsilon\downarrow 0}\limsup_{t\to\infty}
\sup_{\substack{x,y\in K, xt,yt\in\Z^{dp}\\ \|x-y\|\leq \varepsilon}}
\frac1t |\log \Xi^{\WN}_x(0,t)-\log \Xi^{\WN}_y(0,t)|=0.
\end{equation}
To that end, we fix $\varepsilon >0$ and note that for all $t>0$ and all $y\in K$ such that $yt\in\Z^{dp}$,
\begin{equation}
\label{eq:sumchi}
\Xi^{\WN}_y(0,t) = \sum_{\substack{w\in\R^{dp} \\ w(1-\varepsilon)t\in\Z^{dp}}}
\Xi^{\WN}_w(0,(1-\varepsilon)t)\,\Xi^{\WN}_{w,y}((1-\varepsilon)t,t).
\end{equation}
Moreover, by standard large deviation estimates for the number of jumps for each component 
of $\bar{X}^{\kappa}$, it is possible to find an $R>0$ such that 
\begin{equation}
\label{eq:outsideBR}
\limsup_{t\to\infty}\frac1t\log \sup_{y\in K}
\sum_{\substack{w\notin B_R t \\ w(1-\varepsilon)t\in\Z^{dp}}}
\Xi^{\WN}_w(0,(1-\varepsilon)t)\,\,\Xi^{\WN}_{w,y}((1-\varepsilon)t,t) \leq -1,
\end{equation}
so that the main contribution to \eqref{eq:sumchi} comes from those $w$ such that $w\in B_R$. 
Here, $B_R$ denotes the box centered at the origin with radius $R$. Consequently, to conclude 
Step 2 it is enough to show that
\begin{equation}
\label{eq:continuity}
\lim_{\varepsilon\downarrow 0}\limsup_{t\to\infty}
\sup_{\substack{x,y\in K, xt,yt\in\Z^{dp}\\ ||x-y||\leq \varepsilon}}
\frac1t \Bigg|\log \frac{\sum_{w\in B_R t:  w(1-\varepsilon)t\in\Z^{dp}}
\Xi^{\WN}_w(0,(1-\varepsilon)t)\,\, \Xi^{\WN}_{w,x}((1-\varepsilon)t,t)}
{\sum_{w\in B_R t: w(1-\varepsilon)t\in\Z^{dp}}
\Xi^{\WN}_w(0,(1-\varepsilon)t)\,\, \Xi^{\WN}_{w,y}((1-\varepsilon)t,t)}\Bigg|=0.
\end{equation}
But this follows from the fact that the term appearing under the integral in the exponential 
in \eqref{eq:pmwn} is bounded, together with standard estimates on the random walk 
transition kernel. The details can be found in the proof of \cite[Lemma 4.3]{DGRS12}.

\medskip\noindent
{\bf Step 3:} 
Using the results in Steps 1--2, we can conclude the proof as in \cite{DGRS12}. We only
give a sketch. Because of \eqref{eq:cont}, $\alpha$ is continuous and hence can be 
extended to a continuous function $\alpha\colon\,\R^{dp}\to\R$. The uniform convergence 
in \eqref{eq:uniform} follows from \eqref{eq:cont} and a compactness argument. Clearly, 
$\alpha$ is symmetric, i.e., $\alpha(x)=\alpha(-x)$ for all $x\in\R^{dp}$, which is a 
consequence of the symmetry of $\xi$. It remains to show the concavity of $\alpha$. For 
that, fix $x,y\in\R^{dp}$, $\beta\in(0,1)$ and take sequences $(t_n)_{n\in\N}$, $(x_n)_{n\in\N}$,
$(y_n)_{n\in\N}$ such that $\lim_{n\to\infty} t_n=\infty$, $\lim_{n\to\infty} x_nt_n = x$, 
$\lim_{n\to\infty} y_nt_n = y$, and $\beta y_n t_n, (1-\beta)y_nt_n\in\Z^{dp}$ for all $n\in\N$. 
Then, constraining $\bar{X}^{\kappa}$ to be at position $\beta t_n y_n$ at time $\beta t_n$, we see 
that
\begin{equation}
\label{eq:inequchain}
\log \Xi^{\WN}_{\beta y_n+(1-\beta)x_n}(0,t_n) 
\geq \log \Xi^{\WN}_{ y_n}(0,\beta t_n) 
+ \log\Xi^{\WN}_{ y_n , \beta y_n+(1-\beta)x_n}(\beta t_n, t_n).
\end{equation}
The term in the left-hand side converges to $\alpha(\beta y+(1-\beta)x)$ after division by $t_n$, 
the first term in the right-hand side converges to $\alpha(y)$ after division by $\beta t_n$, while 
the second term in the right-hand side converges to $\alpha(x)$ after division by $(1-\beta)t_n$. 
This yields the existence of a function $\alpha$ as claimed in \eqref{eq:uniform}, and finishes 
the proof.
\end{proof}

\subsection{Dynamics {\bf (IIa)}}

\begin{proof}
For $0\leq s < t< \infty$ and $y,z\in \R^{d(n+p)}$ such that $ys,zt\in\Z^{d(n+p)}$, define
\begin{equation}
\label{eq:fsrw}
\Xi^{\FSRW}_{y,z}(s,t) =
(E^{\otimes n}\otimes E^{\otimes p})_{ys}\Bigg(\exp\Bigg\{
\sum_{i=1}^{p}\sum_{j=1}^{n}\int_{0}^{t-s}\one\{X_i^{\kappa}(v)=X_j^{\rho}(v)\}\, dv\Bigg\}
\one\{\bar{X}(t-s)=zt\}\Bigg),
\end{equation}
where $\bar{X}= (X_1^{\kappa},\ldots, X_{p}^{\kappa}, X_1^{\rho},\ldots, X_n^{\rho})$. The function $\alpha$ 
from \eqref{eq:uniform} is constructed on $\R^{d(n+p)}$ rather than on $\R^{dp}$. The construction is 
similar to that for dynamics {\bf (I)} and will therefore be omitted.
\end{proof}

\subsection{Dynamics {\bf (IIb)}}
\medskip

Recall that the dynamics starts from a Poisson random field on $\Z^d$ with intensity $\nu\in (0,\infty)$ and the representation derived in Section~\ref{S2.4.3}.
For the proof we distinguish between two cases. 

\medskip\noindent 
\textbf{Case: $p\geq 1/G(0)$.}

\begin{proof}
For this case it is known that $\lambda_p^{\one}(\kappa)=\infty$ for all choices of $\kappa$, and hence it is 
enough to show that $\lambda_0^{\delta_0}(\kappa)=\infty$. However, this is a simple consequence of 
\cite[Proposition 2.3 and Eq. (3.3)]{GdH06}. 
\end{proof} 

\medskip\noindent
\textbf{Case: $0 < p < 1/G(0)$.}

\begin{proof}
The proof works along similar lines as for {\bf (I)}. We only highlight the differences. For $0\leq s < t <\infty$ 
and $y,t\in\R^{dp}$ such that $ys,zt\in \Z^{dp}$, define
\begin{equation}
\label{eq:Xi2}
\Xi_{y,z}^{\ISRW}(s,t) = 
E_{ys}^{\otimes p}\Bigg(
\exp\Bigg\{\nu \sum_{i=1}^{p} \int_0^{t-s} w(X_i^{\kappa}(v),v)\, dv\Bigg\}\one\{\bar{X}^{\kappa}(t-s)=zt\}\Bigg),
\end{equation}
where $\bar{X}^{\kappa}= (X_1^{\kappa},\ldots, X_p^{\kappa})$ and the process $\bar{X}^{\kappa}$ 
starts at $ys$ under $E_{ys}^{\otimes p}$. Abbreviate $\Xi_y^{\ISRW}(s,t)$ $ = \Xi_{y,y}^{\ISRW}(s,t)$. 
It is again enough to establish a convergence similar to the one in \eqref{eq:uniform}, i.e., to show 
that there is a concave and symmetric function $\alpha\colon\,\R^{dp}\to \R$ such that, for all compact 
subsets $K\subset \R^{dp}$,
\begin{equation}
\label{eq:uniform2}
\lim_{t\to\infty} \sup_{y\in Kt\cap \Z^{dp}} 
\Big|\frac1t \log \Xi_{y/t}^{\ISRW}(0,t)-\alpha(y/t)\Big|=0.
\end{equation}
The proof comes in 3 Steps and is similar to the proof for {\bf (I)}.

\medskip\noindent
{\bf Step 1:}
Define the function $\alpha$ on $\Q^{dp}$ with the help of a superadditivity argument. To exhibit the 
dependence of the function $w$ on the trajectories $X_1^{\kappa},\ldots,X_p^{\kappa}$ we write
\begin{equation}
\label{eq:dep}
w(x,s) = w_{X_{1}^{\kappa}[0,t],\ldots, X_p^{\kappa}[0,t]}(x,s),\quad s\in[0,t].
\end{equation}
It was argued in \cite[Eq. (4.11)]{GdH06} that, for all $s,t\geq 0$,
\begin{equation}
\label{eq:partition}
w_{X_1^{\kappa}[0,s+t],\ldots, X_p^{\kappa}[0,t+s]}(x,u)
\left\{
\begin{array}{ll}
= w_{X_1^{\kappa}[0,s],\ldots, X_p^{\kappa}[0,s]}(x,u), & \mbox{for }u\in[0,s],\\
\geq w_{X_1^{\kappa}[s,s+t],\ldots, X_p^{\kappa}[s,s+t]}(x,u-s), &\mbox{for }u\in[s,s+t].
\end{array}
\right.
\end{equation}
Therefore the superadditivity of $(s,t)\mapsto \log \Xi_{y}^{\ISRW}(s,t)$ follows in a similar fashion 
as for {\bf (I)}. This yields the existence of $\alpha$ on $\Q^{dp}$.

\medskip\noindent
{\bf Step 2:}
As for {\bf (I)}, we want to show that, for any compact subset $K\subset \R^{dp}$,
\begin{equation}
\label{eq:cont2}
\lim_{\varepsilon\downarrow 0}\limsup_{t\to\infty}
\sup_{\substack{x,y\in K, xt,yt\in\Z^{d}\\ ||x-y||\leq \varepsilon}}
\frac1t |\log \Xi^{\ISRW}_x(0,t)-\log \Xi^{\ISRW}_y(0,t)|=0.
\end{equation}
The difference with {\bf (I)} is that we no longer have the same relation as in \eqref{eq:sumchi}. 
However, by the lines following \eqref{eq:w}, we have the bound $w(x,t)\leq \bar{w}(0,t)$ for all 
$x\in\Z^d$, $ t\geq 0$. Moreover, by \eqref{eq:wlimit}, the assumption $0<p<1/G(0)$ yields that 
$\bar{w}(0,t)$ is bounded. Hence, we can use large deviation arguments for the random walk 
to show that the main contribution to \eqref{eq:Xi2} comes from those random walk paths that 
stay until time $t$ inside a box of size $Rt$ for a suitable chosen value of $R$. Moreover, using 
that, for all $t\geq 0$, $\varepsilon\in (0,1)$ and $x,y\in\Z^{dp}$,
\begin{equation}
\label{eq:continuity2}
\begin{aligned}
&E_{0}^{\otimes p}\left(\one\{\bar{X}^{\kappa}((1-\varepsilon) t)=y\}\,
\exp\left\{\sum_{i=1}^{p}\int_0^{(1-\varepsilon) t}w(X_i^{\kappa}(v), v)\, dv\right\}
\right)P_y^{\otimes p}(\bar{X}^{\kappa}(t)= x)\\
&\qquad= E_{0}^{\otimes p}\left(\one\{\bar{X}^{\kappa}((1-\varepsilon) t)=y\}\,
\exp\left\{\sum_{i=1}^{p}\int_0^{(1-\varepsilon)t}w(X_i^{\kappa}(v), v)\, dv\right\}
\one\{\bar{X}^{\kappa}(t)= x\}\right)\\
&\qquad\leq E_{0}^{\otimes p}\left(\one\{\bar{X}^{\kappa}((1-\varepsilon) t)=y\}\,
\exp\left\{\sum_{i=1}^{p}\int_0^{t}w(X_i^{\kappa}(v), v)\, dv\right\}
\one\{\bar{X}^{\kappa}(t)= x\}\right)\\
&\qquad\leq e^{p^2tG(0)/(1-pG(0))}E_{0}^{\otimes p}\left(\one\{\bar{X}^{\kappa}((1-\varepsilon) t)=y\}\,
\exp\left\{\sum_{i=1}^{p}\int_0^{(1-\varepsilon) t}w(X_i^{\kappa}(v), v)\, dv\right\}
\right)\\
&\qquad\qquad\times P_y^{\otimes p}(\bar{X}^{\kappa}(t)= x),
\end{aligned}
\end{equation}
where we used \eqref{eq:wlimit} to obtain the last inequality and the relation~\eqref{eq:partition} 
was used throughout all inequalities in \eqref{eq:continuity2}. We can now proceed as for {\bf(I)}.

\medskip\noindent
{\bf Step 3:} 
This works almost verbatim as for {\bf (I)}. We omit the details. 
\end{proof}

\subsection{Dynamics {\bf (III)}}

\begin{proof}
The idea of the proof is the same as for {\bf (I)--(II)}, but some additional technical difficulties 
arise. Write $\EE_{\mu,x}^{\otimes p}= \EE_{\mu}\otimes E_x^{\otimes p}$ for the expectation 
when $(\xi,\bar{X}^{\kappa})$, with $\bar{X}^{\kappa}=(X_1^{\kappa},\ldots, X_p^{\kappa})$ 
a collection of $p$ indendent simple random walks jumping at rate $2d\kappa$, has initial 
distribution $(\mu,\delta_x)$. For $0\leq s <t <\infty$ and $y,z \in \R^{dp}$ such that $ys,zt
\in\Z^{dp}$, define, similarly as in \eqref{eq:Xi},
\begin{equation}
\label{eq:Xi3}
\Xi_{y,z}^{\Sp}(s,t) = \EE_{\mu,ys}^{\otimes}\Bigg(
\exp\Bigg\{\sum_{i=1}^{p}\int_0^{t-s}\xi(X_i^{\kappa}(v),v)\, dv\Bigg\}
\one\{\bar{X}^{\kappa}(t-s)= zt\}\Bigg),
\end{equation}
and write $\Xi_{y}^{\Sp}(s,t)=\Xi_{y,y}^{\Sp}(s,t)$. As for {\bf (I)}, it is enough to show the 
existence of a function $\alpha\colon\,\R^{dp}\to\R$ such that, for all compact subsets 
$K\subset \R^{dp}$,
\begin{equation}
\label{eq:alpha}
\lim_{t\to\infty}\sup_{y\in Kt\cap\Z^d}\Big|\frac1t\log\Xi_{y/t}^{\Sp}(0,t)-\alpha(y/t)\Big|=0.
\end{equation}
The proof comes in 3 Steps.

\medskip\noindent
{\bf Step 1:} 
We first show the existence of a function $\alpha\colon\,\Q^{dp}\to\R$ such that
\begin{equation}
\label{eq:discretealpha3}
\lim_{\substack{t\to\infty\\ yt\in\Z^{dp}}}\frac1t\log\Xi_{y}^{\Sp}(0,t)=\alpha(y).
\end{equation}
The idea is again to establish the superadditivity of $(s,t)\mapsto \log\Xi_{y}^{\Sp}(s,t)$ for all 
$y\in\Q^{dp}$ such that $ys,yt\in\Z^{dp}$. In the present context, however, this is a bit more tricky 
than before, which is why we provide the details. Fix $y\in\Q^{dp}$, and take $0\leq s < u<t <\infty$ 
such that $ys,yu,yt\in\Z^{dp}$. Constraining the random walk $\bar{X}^{\kappa}$ to be at position $yu$ at 
time $u-s$, we can use the strong Markov property of $(\xi,\bar{X}^{\kappa})$ at time $u-s$ to get
\begin{equation}
\label{eq:smarkov}
\Xi_{y}^{\Sp}(s,t) \geq \EE_{\mu,ys}^{\otimes p}\bigg(\cE(yu,u-s)\,
\EE_{\xi_{u-s},yu}^{\otimes p}\bigg(\cE(yt,t-u)\bigg)\bigg),
\end{equation}
where we abbreviate
\begin{equation}
\label{eq:abbrev1}
\cE(y,t) = \exp\bigg\{\sum_{i=1}^{p}\int_0^t\xi(X_i^{\kappa}(v),v)\, dv\bigg\}\one\{\bar{X}^{\kappa}(t)=y\},
\quad t\geq 0,\, y\in\Z^{dp}.
\end{equation}
Expanding the exponentials, we may rewrite the right-hand side of \eqref{eq:smarkov} as
\begin{equation}
\label{eq:exp}
\begin{aligned}
\sum_{n,m\in\N_0}^{\infty}\frac{1}{n!}\frac{1}{m!}\bigg(\prod_{j=1}^{n}
&\int_0^{u-s}ds_j^{(1)}\bigg)\bigg(\prod_{k=1}^{m}\int_0^{t-u}ds_k^{(2)}\bigg)\\
&\times\EE_{\mu,ys}^{\otimes p}\Big(\cH\big(yu,s_1^{(1)},\ldots, s_n^{(n)};u-s\big)\,
\EE_{\xi_{u-s},yu}^{\otimes p}\Big(\cH\big(yt,s_{1}^{(2)},\ldots, s_{m}^{(2)};t-u\big)\Big)\Big),
\end{aligned}
\end{equation}
where 
\begin{equation}
\label{eq:H}
\cH(y,s_1,\ldots, s_n;t) = \prod_{j=1}^{n}\bigg[\sum_{i=1}^{p}\xi(X_i^{\kappa}(s_j),s_j)\bigg]\,
\one\{\bar{X}^{\kappa}(t)=y\},\quad
n\in\N,\, t, s_1,\ldots, s_n\geq 0,\, y\in\Z^{dp}.
\end{equation}
Note that by the non-negativity of $\xi$, for all $n\in\N$, $yu\in\Z^{dp}$, 
$s_1,\ldots, s_n$, $u-s\geq 0$, 
\begin{equation}
\label{eq:H2}
\cH\big(yu,s_1,\ldots, s_n;u-s\big)
\end{equation}
is a non-decreasing function of the $np$-tuple $(\xi(X_i^{\kappa}(s_j),s_j),\, 1\leq i\leq p,\, 
1\leq j\leq n)$. Hence, the attractiveness of $\xi$ implies that for all $m\in\N$, $yt\in\Z^{dp}$, 
$s_1,\ldots, s_m$, 
$t-u\geq 0$,
\begin{equation}
\label{eq:expH}
\EE_{\xi_{u-s},yu}^{\otimes p}\Big(\cH\big(yt,s_{1},\ldots, s_{m};t-u\big)\Big)
\end{equation}
is a non-decreasing function of $\xi_{u-s}$. Therefore, since $\xi$ is positively correlated
(recall \eqref{eq:poscor}), Liggett~\cite[Corollary 2.21, Section II.2]{L85} yields that
\begin{equation}
\label{eq:Liggett}
\begin{aligned}
\EE_{\mu,ys}^{\otimes p}
&\Big(\cH\big(yu,s_1^{(1)},\ldots, s_n^{(n)};u-s\big)\,\EE_{\xi_{u-s},yu}^{\otimes p}
\Big(\cH\big(yt,s_{1}^{(2)},\ldots, s_{m}^{(2)};t-u\big)\Big)\Big)\\
&= E_{ys}^{\otimes p}\Big[\EE_{\mu}\Big(\cH\big(yu,s_1^{(1)},\ldots, s_n^{(1)};u-s\big)\,
\EE_{\xi_{u-s},yu}^{\otimes p}\Big(\cH\big(yt,s_1^{(2)},\ldots,s_m^{(2)};t-u\big)\Big)\Big)\Big]\\
&\geq E_{yu}^{\otimes p}\Big[\EE_{\mu}\Big(\cH\big(yu,s_1^{(1)},\ldots, s_n^{(1)};u-s\big)\Big)\,
\EE_{\mu\xi_{u-s},yu}^{\otimes p}\Big(\cH\big(yt,s_1^{(2)},\ldots,s_m^{(2)};t-u\big)\Big)\Big],
\end{aligned}
\end{equation}
where $\mu\xi_{u-s}$ is the distribution of $\xi$ at time $u-s$ when $\xi$ starts from $\mu$. 
But $\mu$ is an invariant measure, and so this distribution equals $\mu$. Consequently, the 
right-hand side of \eqref{eq:Liggett} becomes
\begin{equation}
\label{eq:Liggett2}
\EE_{\mu,ys}^{\otimes p}\Big(\cH\big(yu,s_1^{(1)},\ldots, s_n^{(1)};u-s\big)\Big)\,
\EE_{\mu,yu}^{\otimes p}\Big(\cH\big(yt,s_1^{(2)},\ldots,s_m^{(2)};t-u\big)\Big).
\end{equation}
Substituting \eqref{eq:Liggett2} back into \eqref{eq:exp}, we see that
\begin{equation}
\label{eq:superadd2}
\Xi_{y}^{\Sp}(s,t)\geq \Xi_y^{\Sp}(s,u)\,\Xi_y^{\Sp}(u,t),
\end{equation}
from which the existence of $\alpha$ follows.

\medskip\noindent
{\bf Step 2:} 
As in the proof for {\bf (I)}, we want to establish that, for any compact subset $K\subset \R^{d}$,
\begin{equation}
\label{eq:cont3}
\lim_{\varepsilon\downarrow 0}\limsup_{t\to\infty}
\sup_{\substack{x,y\in K, xt,yt\in\Z^{d}\\ ||x-y||\leq \varepsilon}}
\frac1t |\log \Xi^{\Sp}_x(0,t)-\log \Xi^{\Sp}_y(0,t)|=0.
\end{equation}
The difference with {\bf (I)} is that we no longer have the same relation as in \eqref{eq:sumchi}. 
However, because of the boundedness of $\xi$, we can use a large deviation argument for the 
random walk to show that the main contribution to \eqref{eq:Xi3} comes from those random walk 
paths that stay until time $t$ inside a box of size $Rt$ for a suitable chosen value of $R$. Moreover, using 
that, for all $t\geq 0$, $\varepsilon\in(0,1)$ and $w,x\in\Z^{dp}$,
\begin{equation}
\label{eq:continuity3}
\begin{aligned}
&\EE_{\mu,0}^{\otimes p}\left(\one\{\bar{X}^{\kappa}((1-\varepsilon) t)=w\}\,
\exp\left\{\sum_{i=1}^{p}\int_0^{(1-\varepsilon) t}\xi(X_i^{\kappa}(v), v)\, dv\right\}
\right)P_w^{\otimes p}(\bar{X}^{\kappa}(t)= x)\\
&\qquad= \EE_{\mu,0}^{\otimes p}\left(\one\{\bar{X}^{\kappa}((1-\varepsilon) t)=w\}\,
\exp\left\{\sum_{i=1}^{p}\int_0^{(1-\varepsilon)t}\xi(X_i^{\kappa}(v), v)\, dv\right\}
\one\{\bar{X}^{\kappa}(t)= x\}\right)\\
&\qquad\leq\EE_{\mu,0}^{\otimes p}\left(\one\{\bar{X}^{\kappa}((1-\varepsilon) t)=w\}\,
\exp\left\{\sum_{i=1}^{p}\int_0^{t}\xi(X_i^{\kappa}(v), v)\, dv\right\}
\one\{\bar{X}^{\kappa}(t)= x\}\right)\\
&\qquad\leq e^{p\varepsilon t}\EE_{\mu,0}^{\otimes p}\left(\one\{\bar{X}^{\kappa}((1-\varepsilon) t)=w\}\,
\exp\left\{\sum_{i=1}^{p}\int_0^{(1-\varepsilon) t}\xi(X_i^{\kappa}(v), v)\, dv\right\}
\right)P_w^{\otimes p}(\bar{X}^{\kappa}(t)= x),
\end{aligned}
\end{equation}
we can finish the proof as for {\bf (I)}.

\medskip\noindent
{\bf Step 3:}
Use the techniques from Step 1 to proceed in a similar manner as in Step 3 for {\bf (I)}. 
We omit the details.
\end{proof}

\section{A technical lemma}
\label{SB}

The following lemma was used in Section~\ref{S2.4.3}.

\begin{lemma}
\label{lem:B}
Let $L$ be the generator of the dynamics in {\bf (IIb)}. For $N\in\N$, define $V_N\colon\,\N_0^{\Z^d}
\times\Z^d\to\R$ by $V_N(\eta,x) = \eta(x) \wedge N$ (recall \eqref{VNdef}), and let $P_t^{V_N}$ 
be the semigroup of $\cL^{V_N} = L+\Delta^{\cK} + V_N$. Then for every $t>0$ there is a 
$g\in L^1(\N_0^{\Z^d}\times\Z^d,\mu\otimes m)$ such that, for all $\eta\in \N_0^{\Z^d}$ and 
$y\in\Z^d$,
\begin{equation}
\label{eq:B}
\left|\left(\cL^{V_N}P_t^{V_N}\bar{f}\right)(\eta,y)\times
\left(P_t^{V_N}\bar{f}\right)(\eta,y)\right| \leq g(\eta,y)
\end{equation}
locally uniformly in $t$. Here, for $R>0$, $\bar{f}(\eta, y)=\one\{y\in B_R\}$ and $B_R$ is the 
box centered around the origin with radius $R$.
\end{lemma}

\begin{proof}
We may assume that $t\geq 1$, which we do for notational convenience. It is straightforward to 
show that the statement is true when $\cL^{V_N}= L+\Delta^{\cK}+V_N$ in \eqref{eq:B} is 
replaced by $\Delta^{\cK}+V_N$. Furthermore, since $B_R$ is a finite set, it is enough to show that
\begin{equation}
\label{eq:B2}
\left|\left(L P_t^{V_N}\delta_w\right)(\eta,y)\times\left(P_t^{V_N}\delta_v\right)(\eta,y)\right|
\leq g(\eta,y)
\end{equation}
for any $v,w\in\Z^d$. For notational convenience we assume that $v=w=0$. The general case follows 
in a similar manner. Note that by the definition of $L$ in \eqref{eq:genISRW} we see that, for all $t>0$ 
and $(\eta,y)\in \N^{\Z^d}\times\Z^d$,
\begin{equation}
\label{eq:Bgenest}
\begin{aligned}
&\left|\left(L P_t^{V_N}\delta_0\right)(\eta,y)\times\left(P_t^{V_N}\delta_0\right)(\eta,y)\right|\\
&\quad \leq \sum_{x\in\Z^d} \,\,\sum_{z\colon\|z\|=1} \eta(x)
\,\left|\left(P_t^{V_N}\delta_0\right)(\eta^{x,x+z},y)-\left(P_t^{V_{N}}\delta_0\right)(\eta,y)\right|
\,e^{Nt}\,P_y(X_t^{\cK}=0).
\end{aligned}
\end{equation}
To estimate the difference $|(P_t^{V_{N}}\delta_0)(\eta^{x,x+z},y)-(P_t^{V_{N}}\delta_0)(\eta,y)|$, 
we introduce the following coupling.

Let $x\in\Z^d$ such that $\eta(x)\geq 1$, and let $\xi$ be {\bf (IIb)} started in $\eta$ and 
$\xi^{x,x+z}$ be {\bf (IIb)} started in $\eta^{x,x+z}$. Note that both systems start with the 
same number of simple random walks. Let $Y_x$ be simple random walk with jump rate 
$2d$ started from $x$. We can couple $\xi$ and $\xi^{x,x+z}$ such that, for all $w\in\Z^d$ 
and $t\geq 0$,
\begin{equation}
\label{eq:couplexi}
\xi^{x,x+z}(w,s) =
\left\{ 
\begin{array}{ll}
\xi(w,s), &w\neq Y_x(s), Y_x(s)+z,\\
\xi(w,s)-1, &w=Y_x(s),\\
\xi(w,s)+1, &w=Y_x(s)+z.
\end{array}\right.
\end{equation}
With this coupling at hand, we see that
\begin{equation}
\label{eq:Bdifest}
\begin{aligned}
&\left|\left(P_t^{V_{N}}\delta_0\right)(\eta^{x,x+z},y)
-\left(P_t^{V_{N}}\delta_0\right)(\eta,y)\right|\\
&\quad\leq 2e^{Nt}P_{y,x}(\exists s\in[0,t]\colon\, X^{\cK}(s)\in \{Y_x(s),Y_x(s)+z\},
X^{\cK}(t)=0),
\end{aligned}
\end{equation}
where $P_{y,x}$ denotes the product measure of $(X^{\cK},Y_x)$ started from $(y,x)$.
Combining \eqref{eq:Bgenest} and \eqref{eq:Bdifest}, we see that
\begin{equation}
\label{eq:Bgenestfin}
\begin{aligned}
&\sup_{s\in[t-1,t+1]} \left|\left(L P_s^{V_N}\delta_0\right)(\eta,y)
\left(P_s^{V_N}\delta_0\right)(\eta,y)\right|\\
&\quad\leq 2e^{2N(t+1)}\\
&\quad\times\sum_{x\in\Z^d}\,\,\sum_{z\colon\|z\|=1}\eta(x)
P_{y,x}\Big(\exists\, s\in[0,t+1]\colon\, \\
&\qquad\qquad  X^{\cK}(s) \in \{Y_x(s),Y_x(s)+z\}, 
0\in X^{\cK}([t-1,t+1])\Big)\\
&\quad \times P_y\left(\exists\, s\in[t-1,t+1]\colon\, X^{\cK}(s)=0\right).
\end{aligned}
\end{equation}

We complete the proof by showing that the right-hand side of \eqref{eq:Bgenestfin} is 
in $L^1(\N_0^{\Z^d}\times\Z^d,\mu\otimes m)$. To see why, note that integration of the 
right-hand side of \eqref{eq:Bgenestfin} over $\eta\in \N_0^{\Z^d}$ yields the upper bound
\begin{equation}
\label{eq:BL1}
\begin{aligned}
2\sum_{x\in\Z^d}\,\,\sum_{z\colon\|z\|=1} \nu 
&P_{y,x}\left(\exists s\in[0,t+1]\colon\, X^{\cK}(s)\in \{Y_x(s),Y_x(s)+z\},
0\in X^{\cK}([t-1,t+1])\right)\\
&\qquad \times P_y\left(\exists\, s\in[t-1,t+1]\colon\, X^{\cK}(s)=0\right).
\end{aligned}
\end{equation}
Summing the right-hand side of \eqref{eq:BL1} over $y\in\Z^d$, we get
\begin{equation}
\label{eq:BL1est}
\begin{aligned}
&2 e^{2N(t+1)}\, \nu \sum_{x,y\in\Z^d} 
P_{y,x}\left(\exists s\in[0,t+1]\colon\,||X^{\cK}(s)-Y_x(s)||=1, 0\in X^{\cK}([t-1,t+1])\right)\\
&\qquad\qquad\times P_y\left(\exists\, s\in[t-1,t+1]\colon\, X^{\cK}(s)=0\right)
\leq I + II,
\end{aligned}
\end{equation}
where 
\begin{equation}
\label{eq:BIandII}
\begin{aligned}
I &= 2 e^{2N(t+1)}\,\nu \sum_{y\in\Z^d} 
P_y\left(\exists\, s\in[t-1,t+1]\colon\, X^{\cK}(s)=0\right),\\
II &= 2 e^{2N(t+1)}\,\nu \sum_{x\neq y} 
P_{y,x}\left(\exists s\in[0,t+1]\colon\, \|X^{\cK}(s)-Y_x(s)\|=1,0\in X^{\cK}([t-1,t+1])\right).
\end{aligned}
\end{equation}
The first property in \eqref{staterg} combined with standard large deviations estimates shows that 
$I$ is finite. To see that $II$ is finite, note that by the Cauchy-Schwarz inequality,
\begin{equation}
\label{eq:BII}
\begin{aligned}
II \leq 2 e^{2N(t+1)} \sum_{y\in\Z^d} \sum_{k\in\N} \,\, \sum_{x\colon\|x-y\|=k} 
&\sqrt{P_{y,x}\Big(\exists s\in[0,t+1]\colon
\|X^{\cK}(s)-Y_x(s)\|=1\Big)}\\
&\times\sqrt{P_y\Big(0\in X^{\cK}([t-1,t+1])\Big)}.
\end{aligned}
\end{equation}
To proceed, note that for any $y\in\Z^d$,
 \begin{equation}
\label{eq:BIIinner}
\begin{aligned}
 \sum_{k\in\N} \,\, \sum_{x\colon\|x-y\|=k} 
&\sqrt{P_{y,x}\Big(\exists s\in[0,t+1]\colon
\|X^{\cK}(s)-Y_x(s)\|=1\Big)}\\
&\leq \sum_{k\in\N} \,\, \sum_{x\colon\|x-y\|=k} 
\sqrt{P_{y,x}\Big(N(X^\mathcal{K}-Y_x,t+1)\geq k-1\Big)}.
\end{aligned}
\end{equation}
Here, $N(X^{\cK}-Y_{x},t+1)$ denotes the number of jumps of $X^{\cK}-Y_{x}$. Thus, the 
first property in \eqref{staterg} combined with standard large deviation estimates shows 
that the sum in \eqref{eq:BIIinner} is bounded uniformly in $y$. To conclude, use that 
\begin{equation}
y\mapsto \sqrt{P_y(0\in X^{\cK}([t-1,t+1]))} 
\end{equation}
decays faster than exponential in $\|y\|$. This implies that $II$ is finite, and shows 
that the right-hand side of \eqref{eq:Bgenestfin} is in $L^1(\N_0^{\Z^d}\times\Z^d,\mu\otimes m)$.
\end{proof}


\end{document}